\documentclass[10pt]{amsart}

\usepackage{amsrefs}
\usepackage{amsmath}
\usepackage{amssymb}
\usepackage{amsthm}
\usepackage{colonequals}
\usepackage{enumerate}
\usepackage[T3,T1]{fontenc}
\usepackage[margin=1in]{geometry}
\usepackage{hyperref}
\usepackage{xcolor}

\DeclareSymbolFont{tipa}{T3}{cmr}{m}{n}
\DeclareMathAccent{\invbreve}{\mathalpha}{tipa}{16}

\def\dime{{\mathfrak{Deg}}}
\def\h{{\bf h}}
\def\LL{{{Length}}}
\def\nn{\mathbb{N}}
\def\O{\Omega}
\def\rr{\mathbb{R}}
\def\bve{\bar\varepsilon}

\newcommand{\mM}{{\mathcal M}}
\newcommand{\mN}{{\mathcal N}}

\newcommand{\appsection}[1]{\let\oldthesection\thesection
\renewcommand{\thesection}{\oldthesection}
\section{#1}\let\thesection\oldthesection}

\newtheorem{theorem}{Theorem}[section]
\newtheorem{proposition}[theorem]{Proposition}
\newtheorem{lemma}[theorem]{Lemma}
\newtheorem{corollary}[theorem]{Corollary}
\newtheorem{definition}[theorem]{Definition}
\theoremstyle{definition}
\newtheorem{example}[theorem]{Example}
\newtheorem{examples}[theorem]{Examples}
\newtheorem{remark}[theorem]{Remark}

\let\d=\delta
\let\G=\Gamma
\let\L=\Lambda
\let\o=\omega
\let\ve=\varepsilon

\allowdisplaybreaks

\title{An $L^\infty$ Rashevskii-Chow Theorem}

\author{Ermal Feleqi}
\address{Department of Mathematics, University of Vlor{\"e}, Vlor{\"e}, Albania}
\email{ermal.feleqi@univlora.edu.al}

\author{Rohit Gupta}
\address{Institute for Mathematics and its Applications, University of Minnesota, Minneapolis, MN 55455, USA (Past Affiliation); Department of Aerospace Engineering, IIT Bombay, Powai, Mumbai, Maharashtra 400076, India (Present Affiliation)}
\email{rohit@aero.iitb.ac.in}

\author{Franco Rampazzo}
\address{Dipartimento di Matematica ``Tullio Levi-Civita'', Universit{\`a} degli Studi di Padova, Via Trieste, 63, Padova 35121, Italy}
\email{rampazzo@math.unipd.it}

\begin{document}

\maketitle

\begin{abstract}
Consider a finite family $\{f_1,\hdots,f_\nu\}$ of $C^\infty$ vector fields on a $n$-dimensional ($n\in\nn$), smooth manifold $\mM$. The celebrated Rashevskii-Chow theorem states that, provided the vector fields $\{f_1,\hdots,f_\nu\}$, together with their iterated Lie brackets, span the whole tangent space at some $x_*\in\mM$, then any $x$ in a neighborhood of $x_*$ can be connected to $x_*$ by means of a finite concatenation of {integral curves of $\{\pm f_1,\hdots,\pm f_\nu\}$}. This result finds applications in a number of areas, e.g., in control theory, in Sub-Riemannian geometry, and the theory of degenerate elliptic and parabolic partial differential equations, to mention a few. Here we extend this basic result to families of vector fields, which are considerably less regular, in particular, by allowing iterated Lie brackets to be just bounded measurable. This is technically made possible by the utilization of set-valued Lie brackets, which have already proven to be useful in extending commutativity type results, Frobenius' theorem, and also higher-order necessary conditions for optimal control problems, to the setting of non-smooth vector fields.
\end{abstract}

\section{Introduction}\label{s-intro}
\subsection{Problem Statement}
Let {$\{f_1,\hdots,f_\nu\}$ be a finite family of vector fields on a $n$-dimensional ($n\in\nn$), smooth manifold $\mM$. Consider the control system:
\begin{equation}\label{sistemaintro1}
\dot x=\displaystyle\sum_{i=1}^\nu u_if_i(x),
\end{equation}
where for every time $t>0$, the map $u=(u_1,\hdots,u_\nu)\colon [0,t]\to\rr^\nu$ is called an \emph{admissible control}, if it is piecewise constant and takes values in the set of standard basis vectors $\{\pm {\bf e}_1,\hdots,\pm {\bf e}_\nu\}$ of $\rr^\nu$. In particular, the solutions $x(\cdot)$ of \eqref{sistemaintro1} corresponding to admissible controls are nothing but finite concatenations of forward and backward integral curves of the vector fields $\{f_1,\hdots,f_\nu\}$. Let us assume that the vector fields $\{f_1,\hdots,f_\nu\}$ are locally Lipschitz continuous (this assumption is further relaxed in Section~\ref{s-Chow}). Then, for a given point $x_*\in\mM$, a time $t>0$ sufficiently small, and an admissible control $u$, let us denote by $x_{x_*,u}\colon [0,t]\to\mM$, the unique solution of \eqref{sistemaintro1} starting from $x(0)=x_*$ and corresponding to the control $u$.} Let us denote the \emph{reachable set from $x_*$ up to time $t$}, by $\text{\textsf{Reach}}(t,x_*)\subset \mM$, which is defined as follows:
\begin{equation*}
\text{\textsf{Reach}}(t,x_*)\colonequals\bigcup_{0\leq s\leq t}\{x_{x_*,u}(s)\; : \;{u}\colon {[0,s]}\to\rr^\nu~\text{is an admissible control}\}.
\end{equation*}
One says that the control system \eqref{sistemaintro1} is \emph{small}-\emph{time locally controllable} (STLC) from $x_*$, if for every sufficiently small time $t>0$, the {reachable set} $\text{\textsf{Reach}}(t,x_*)$ is a neighborhood of $x_*$. Let us also point out that the interest in STLC goes far beyond just the realm of control theory, as this notion is also foundational in \emph{Sub}-\emph{Riemannian geometry}, in the theory of \emph{degenerate elliptic} and \emph{parabolic partial differential equations}, as well as in the study of \emph{eikonal equations} (see Remark~\ref{literature} for some classic references in these areas).

The celebrated Rashevskii-Chow theorem (see \cite{Caratheodory1909,Rashevskii1938,Chow1939}) provides a sufficient condition in order that \eqref{sistemaintro1} be STLC. To state it, let us recall that the \emph{Lie bracket} $[g_1,g_2]$ of two {differentiable} vector fields $g_1$ and $g_2$, is itself a vector field (i.e., a first-order partial differential operator), which, in any system of local coordinates, is defined as $[g_1,g_2]\colonequals Dg_2\cdot g_1-Dg_1\cdot g_2$. Provided the vector fields $\{f_1,\hdots,f_\nu\}$ are sufficiently regular, one can iterate the bracketing operation a finite number of times, thereby obtaining \emph{iterated} Lie brackets. If the vector fields $\{f_1,\hdots,f_\nu\}$ are of class $C^\infty$ (as is customary, the notation $C^k$, for an integer $k\geq 0$ or $k=\infty$, stands for $k$-{times continuously differentiable}), one uses $\text{\textsf{Lie}}(\{f_1,\hdots,f_\nu\})$ to denote the $C^\infty(\mM,\rr)$-algebra generated by $\{f_1,\hdots,f_\nu\}$, where the bracketing operation is regarded as the ``multiplication'' in this algebra. We are now ready to state the classic Rashevskii-Chow theorem, where for any given point $x\in\mM$, we let
\begin{equation*}
\text{\textsf{Lie}}(\{f_1,\hdots,f_\nu\})(x)\colonequals\{v(x)\; : \;v\in\text{\textsf{Lie}}(\{f_1,\hdots,f_\nu\})\}.
\end{equation*}

\begin{theorem}[$C^\infty$ Rashevskii-Chow theorem]\label{CRintro1}
Let $\nu\geq 1$ be a given integer. If the family $\{f_1,\hdots,f_\nu\}$ of $C^\infty$ vector fields on a finite-dimensional smooth manifold $\mM$ is \emph{bracket-generating} at {$x_*\in \mM$}, i.e.,
\begin{equation}\label{bragen1}
\emph{\textsf{Lie}}(\{f_1,\hdots,f_\nu\})(x_*)=T_{x_*}\mM,
\end{equation}
where $T_{x_*}\mM$ denotes the tangent space to the manifold $\mM$ at $x_*\in\mM$, then the control system \eqref{sistemaintro1} is STLC from $x_*$.
\end{theorem}

In addition, one has Theorem~\ref{CRintro2} below regarding the {\it minimum-time function} $x\mapsto T(x)$. The latter is defined as the smallest time $t\geq 0$ such that the solution $x_{x_*,u}(\cdot)$ of \eqref{sistemaintro1}, starting from $x(0)=x_*$ and corresponding to an admissible control $u=(u_1,\hdots,u_\nu)\colon [0,t]\to\{\pm {\bf e}_1,\hdots,\pm {\bf e}_\nu\}$, satisfies $x(t)=x$. Moreover, in the statement of Theorem~\ref{CRintro2}, by saying that ``\eqref{bragen1} is verified at step $r$'', we mean that the condition $\text{\textsf{Lie}}(\{f_1,\hdots,f_\nu\})(x_*)=T_{x_*}\mM$ is achieved by iterated Lie brackets of {length} less than or equal to $r$, where the \emph{length} of a Lie bracket is the total number of vector fields appearing in it, including possible repetitions (for instance, the Lie brackets $[f_3,f_4]$ and $[f_3,[f_2,[f_5,f_2]]]$ have lengths $2$ and $4$, respectively; in addition, each vector field $f$ is considered an iterated Lie bracket of length one).
\begin{theorem}\label{CRintro2}
Let the hypotheses of Theorem~\ref{CRintro1} hold. If \eqref{bragen1} is verified at \emph{step} $r$, then the \emph{minimum}-\emph{time function} $T\colon \mathcal{O}\to [0,+\infty]$ verifies $T(x)\leq Cd(x,x_*)^{1/r}$ for all $x$ in some neighborhood $\mathcal{O}$ of $x_*$ (endowed with a Riemannian distance $d\colon\mathcal{O}\times\mathcal{O}\to\rr$), with $C>0$ being a suitable constant.
\end{theorem}

A natural question that now arises is to what extent can one weaken the regularity hypothesis stated in Theorem~\ref{CRintro1}. For instance, could the vector fields be assumed to be of class $C^1\cap C^{r-1}$, for some $r\geq 1$, provided that the bracket generating property is verified at step $r$? In particular, such a regularity hypothesis would guarantee that iterated Lie brackets of length less than or equal to $r$ are continuous. In fact, a {Rashevskii-Chow theorem of this type} has been proved in \cite{Bramanti2013}. However, since the local flow of a vector field $g$ exists and is uniquely determined as soon as $g$ is locally Lipschitz continuous, it is reasonable to wonder if some further generalization of the Rashevskii-Chow theorem is still possible by employing an appropriate notion of a ``Lie bracket'' of non-differentiable vector fields. A partial step in this direction (which is limited to Lie brackets of length $2$) has been taken in \cite{Rampazzo2001}, where Theorem~\ref{RC Theorem special} below was proved. To state it, let us introduce the notion of a \emph{set}-\emph{valued Lie bracket} $[g_1,g_2]_{\text{set}}$ of locally Lipschitz continuous vector fields $g_1$ and $g_2$, defined for every $x\in \mM$ as follows:
\begin{equation*}
[g_1,g_2]_{\text{set}}(x)\colonequals\text{\bf co}\left\{v\; : \;v=\lim_{k\to\infty}[g_1,g_2](x+h_k)\right\},
\end{equation*}
where for any subset $W$ of a vector space $V$, the set ${\bf co}~W$ denotes its \emph{convex hull}, and the limit is taken along all sequences $\{h_{k}\}_{k\in\nn}$ converging to ${\bf 0}$ and such that $\{x+h_k\}_{k\in\nn}\subset\textsf{Diff}^{1}(g_1)\cap\textsf{Diff}^{1}(g_2)$ (here, if $g$ is a vector field, $\textsf{Diff}^{1}(g)$ denotes the set of the points where $g$ is differentiable). Observe that, by Rademacher's theorem, the standard single-valued Lie bracket $[g_1,g_2](x)$ is defined for almost every $x\in\mM$, while the set-valued Lie bracket $[g_1,g_2]_{\text{set}}(x)$ is well-defined for every $x\in\mM$, and takes non-empty compact and convex values. Let us also point out that this notion of a set-valued Lie bracket has already proven to be useful in extending --to the setting of non-smooth vector fields (see \cite{Rampazzo2007commutators,Rampazzo2007Frobenius,Angrisani2023})-- commutativity type results, Frobenius' theorem, and also higher-order necessary conditions for optimal control problems. In this setting, the Rashevskii-Chow type result reads as follows: 
\begin{theorem}[$L^\infty$ Rashevskii-Chow theorem for Lie brackets of length 2 \cite{Rampazzo2001}]\label{RC Theorem special}
Let $\{f_1,\hdots,f_\nu\}$ be a family of locally Lipschitz continuous vector fields on a manifold $\mM$, and for some $x_*\in\mM$ and every choice of $v_{k\ell}\in [f_k,f_\ell]_{\emph{set}}(x_*)$, $k,\ell=1,\hdots,\nu$, let the following \emph{bracket-generating condition} be satisfied:
\begin{equation*}
\emph{span}\left\{f_j(x_*),v_{k\ell}\; : \;j\in\{1,\hdots,\nu\},~(k,\ell)\in\{1,\hdots,\nu\}\times\{1,\hdots,\nu\}\right\}=T_{x_*}\mM.
\end{equation*}
Then, the control system \eqref{sistemaintro1} is STLC from $x_*$. Furthermore, the minimum-time function $T\colon \mathcal{O}\to [0,+\infty]$, verifies $T(x)\leq C\,d(x,x_*)^{1/2}$ for all $x$ in some neighborhood $\mathcal{O}$ of $x_*$ (endowed with a Riemannian distance $d\colon\mathcal{O}\times\mathcal{O}\to\rr$), with $C>0$ being a suitable constant.\\
\emph{(See \cite{Rampazzo2001} for a slightly more general version of this theorem, involving also vector fields that are merely continuous).}
\end{theorem}

In this paper we prove a generalization of Theorem~\ref{RC Theorem special} including set-valued Lie brackets of {\it any length} that are single-valued almost everywhere, besides being locally bounded and measurable. Such a program immediately poses the following two crucial questions:
\begin{enumerate}[{\bf (Q1)}]
\item\emph{How to define a set-valued iterated Lie bracket of length greater than $2$?}
\item\emph{Provided {\bf (Q1)} is answered reasonably, is there some \emph{open mapping result} available, which allows us to obtain a generalization of Theorem~\ref{RC Theorem special}?}
\end{enumerate}

As for question {\bf (Q1)}, it is obviously crucial that the definition of a set-valued Lie bracket should allow for an {asymptotic formula} for the associated {multi}-{flow}. For instance, the Lie brackets $[g_2,g_3]_{\text{set}}$ and $[g_1,[g_2,g_3]]_{\text{set}}$ should be ``approximations'' (in a sense that we make precise subsequently) of the multi-flows:
\begin{equation*}
\Psi(t_2,t_3)(x)\colonequals e^{-t_3g_3}\circ e^{-t_2g_2}\circ e^{t_3g_3}\circ e^{t_2g_2}(x) 
\end{equation*}
and
\begin{equation*}
\Psi(t_1,t_2,t_3)(x)\colonequals\left(e^{-t_2g_2}\circ e^{-t_3g_3}\circ e^{t_2g_2}\circ e^{t_3g_3}\right)\circ e^{-t_1g_1} \circ \left(e^{-t_3g_3}\circ e^{-t_2g_2}\circ e^{t_3g_3}\circ e^{t_2g_2}\right)\circ e^{t_1g_1}(x),
\end{equation*}
respectively, where the mapping $\rr\times\mM\ni(t,x)\mapsto e^{tg}(x)\in\mM$ denotes the flow of a (sufficiently regular) vector field $g$. As a first guess, one could proceed by induction on the length of the Lie bracket. For instance, if the vector field $g_1$ were locally Lipschitz continuous and the vector fields $g_2,g_3$ were differentiable with locally Lipschitz continuous derivatives, then one might want to set {$[g_1,[g_2,g_3]]_{\text{set}}\colonequals [g_1,g]_{\text{set}}$}, where {the vector field $g$} coincides with the locally Lipschitz continuous vector field $[g_2,g_3]$. In fact, this approach does not work, as it does not allow for a satisfactory asymptotic formula (see \cite[Section~7]{Rampazzo2007commutators} for a counterexample). However, a positive answer for this issue, agreeing with the need of a well-functioning asymptotic formula, has been proposed in \cite{Feleqi2017}. Deferring the definition of a set-valued Lie bracket to Section~\ref{s-brackets} in the general case, let us only define this for a particular case of a bracket of length $3$:
\begin{align*}
&[[g_1,g_2],g_3]_{\text{set}}(x)\colonequals\\
&\text{\bf co}\left\{v\; : \;v=\lim_{k\to\infty}Dg_3(x+h_{1k})\cdot [g_1,g_2](x+h_{2k})-D[g_1,g_2](x+h_{2k})\cdot g_3(x+h_{1k})\right\}.
\end{align*}
In the above definition, the limit is taken along all sequences $\{(h_{1k},h_{2k})\}_{k\in\nn}$ converging to ${\bf 0}$ such that $\{(x+h_{1k},x+h_{2k})\}_{k\in\nn}\subset\textsf{Diff}^{1}(g_{3})\times\textsf{Diff}^{1}([g_{1},g_{2}])$. For an integer $r\geq 0$, let us use the notation $C^{r,1}$ to denote the class of vector fields that are $r$-times continuously differentiable and whose $r$-$th$ order derivatives are locally Lipschitz continuous. Hence, in the above definition we are assuming that $g_1,g_2\in C^{1,1}$ and $g_3\in C^{0,1}$. Let us carefully note that in the above definition, \emph{two independent sequences of points} have been used instead of just a \emph{single one} (as it would have been the case, if one merely used an inductive procedure to define a set-valued Lie bracket of length $3$). Actually, the number of independent sequences of points to be used in defining a set-valued Lie bracket in the general case depends on its ``structure'' rather than its length, and this is explained rigorously in Section~\ref{s-brackets}. The case of set-valued Lie brackets of length $4$ already highlights this point since, for instance, the definition of the set-valued Lie bracket $[[g_1,g_2],[g_3,g_4]]_{\text{set}}$ uses two independent sequences of points, while the definition of the set-valued Lie bracket $[g_1,[g_2,[g_3,g_4]]]_{\text{set}}$ uses three such ones (see Examples~\ref{examples}). Moreover, to deal with the general case of a set-valued Lie bracket of any finite length greater than $2$, so as to ensure that the resulting vector field is single-valued almost everywhere, locally bounded and measurable, we need the notion of a ``formal iterated bracket'' $B$ and the associated concept of ``$C^{B-1,1}$-regularity'' of a finite family of vector fields (see Section~\ref{s-brackets}). As proved in \cite{Feleqi2017}, these iterated set-valued brackets are chart-independent and a suitable asymptotic formula holds for them, recalled here in Theorem~\ref{as-th}.

As for question {\bf (Q2)}, we rely on the theory of \emph{generalized differential quotients} (GDQs), introduced by H. J. Sussmann in \cite{Sussmann2000new, Sussmann2000resultats, Sussmann2008}. To apply the chain rule and an open mapping result valid within the framework of this theory, we need to show that a set-valued Lie bracket (of any finite length) is indeed a GDQ of the associated multi-flow. This involves combining the said asymptotic formula with a sufficient condition ensuring that a specific subset of suitable linear mappings is a GDQ. The said sufficient condition is proved in Theorem~\ref{sucole}.

Relying partly on intuition and referring the reader to the main body of the paper, we now present a general \emph{bracket}-\emph{generating condition}, under which the main result of this paper is proved.\footnote{For the sake of brevity, here and also in the title of the paper, we write ``$L^\infty$'' to mean \emph{locally bounded and measurable}, although a more correct notation to use would be ``$L_{\text{loc}}^\infty$''.} In this definition, the notation $\dime(B)$ stands for the \emph{highest degree of differentiation involved in the formal iterated bracket $B$} (see Section~\ref{s-brackets}).
\begin{definition}[$L^\infty$ bracket-generating condition]\label{brackgenintro}
Let $B_1,\hdots,B_\ell$ be formal {iterated} brackets. For a given integer $\nu\geq 1$, we say that a finite family $\{g_1,\hdots,g_\nu\}$ of vector fields on a finite-dimensional manifold $\mM$ {of class $C^q$, for some $q\geq\max\left\{1,\dime(B_j)\; : \;j\in\{1,\hdots,\ell\}\right\}$,} is \emph{$L^\infty$ bracket}-\emph{generating at a point ${x}\in \mM$ with respect to the formal brackets $B_1,\hdots,B_\ell$}, if there exists an integer $r\geq 1${, with $r\leq\nu$,} such that the $r$-tuple $\hat{\mathbf g}=(g_1,\hdots,g_r)$ verifies the following conditions:
\begin{enumerate}[(i)]
\item{For every $j\in\{1,\hdots,{\ell}\}$, one has
\begin{equation*}
{\bf\hat g}\in\begin{cases}
C^{B_j}, &\emph{if length of $B_j=1$},\\
C^{B_j-1,1}, &\emph{if length of $B_j>1$};
\end{cases}
\end{equation*}}
\item{For {every} {$\ell$-tuple} $(v_1,\hdots,v_\ell)\in {(B_1)}_\emph{set}({\bf\hat g})({x})\times\cdots\times {(B_\ell)}_\emph{set}({\bf\hat g})({x})$, one has
\begin{equation*}
\emph{span}\left\{v_1,\hdots,v_\ell\right\}=T_x\mM.
\end{equation*}}
\end{enumerate}
\end{definition}

We are now ready to state the main result (see Theorem~\ref{ChowRas}):
\begin{theorem}[$L^\infty$ Rashevskii-Chow theorem]\label{RC Theorem general}
Let $B_1,\hdots,B_\ell$ be formal iterated brackets. For a given integer $\nu\geq 1$, if the family $\{f_1,\hdots,f_\nu\}$ of vector fields on a finite-dimensional manifold $\mM$ {of class $C^q$, for some integer $q\geq\max\left\{1,\dime(B_j)\; : \;j\in\{1,\hdots,\ell\}\right\}$,} is $L^\infty$ bracket-generating at $x_*\in\mM$ with respect to the formal brackets $B_1,\hdots,B_\ell$, then the control system \eqref{sistemaintro1} is STLC from $x_*$. Furthermore, the minimum-time function $T\colon \mathcal{O}\to [0,+\infty]$ verifies $T(x)\leq Cd(x,x_*)^{1/r}$ for all $x$ in some neighborhood $\mathcal{O}$ of $x_*$ (endowed with a Riemannian distance $d\colon\mathcal{O}\times\mathcal{O}\to\rr$), with $C>0$ being a suitable constant and ${r}\colonequals\max\left\{\emph{\textsf{\LL}}(B_j)\; : \;j\in\{1,\hdots,\ell\}\right\}$.
\end{theorem}

We refer the reader to Example~\ref{esempio} for an application of Theorem~\ref{RC Theorem general}. Finally, in Theorem~\ref{wChowRas} we state a generalization of Theorem~\ref{RC Theorem general}, denominated as ``A non-deterministic version of the $L^\infty$ Rashevskii-Chow theorem'', and concerns the case when some of the vector fields $\{f_1,\hdots,f_\nu\}$ are merely continuous (so that a lack of uniqueness of solutions for the corresponding Cauchy problems may occur).

\section{Set-Valued Iterated Lie Brackets}\label{s-brackets}
 In this section, we recall the notion and the main properties of the set-valued iterated Lie bracket introduced in \cite{Feleqi2017}.

\subsection{Formal Brackets}\label{formalsec}
Given a sequence ${\bf X}=\{X_j\}_{j\in\nn}$ of distinct objects called \emph{indeterminates} or \emph{variables}, let $W({\bf X})$ {denote} the set of all \emph{words}{, where a word is a finite ordered string that can be obtained by using the variables $X_j$, with $j\in\nn$}, the left (square) bracket, the right (square) bracket, and a comma. The \emph{letter sequence} of a word {$w\in W({\bf X})$, denoted by $\textsf{Seq}(w)$, is the sequence} obtained from $w$ by deleting all the brackets and commas. The {\emph{length} of a word $w\in W({\bf X})$, denoted by $\text{\textsf{\LL}}(w)$,} is the total number of variables appearing in {$\textsf{Seq}(w)$}.

We are primarily concerned with a special type of words {called} \emph{formal iterated brackets}, or sometimes, for the sake of brevity, called only \emph{formal brackets}. A single variable is also called a \emph{formal bracket of length $1$}. We define the formal iterated brackets, recursively as follows. We say that, if $B_1$ and $B_2$ are formal brackets, the word {$[B_1,B_2]$} is the \emph{formal iterated bracket with factorization} $(B_1,B_2)$. Note that $\text{\textsf{\LL}}([B_1,B_2])=\text{\textsf{\LL}}(B_1)+\text{\textsf{\LL}}(B_2)$. Hence, the length of a bracket is equal to the number of commas plus 1. Any sub-string of a formal bracket $B$, which itself is a formal iterated bracket, is called a \emph{sub}-\emph{bracket} of $B$. We adopt the following convention:
\begin{enumerate}[{$\bullet$}]
\item{If $B$ is a formal bracket, the sequence $\textsf{Seq}(B)$ is the word obtained by $B$ by deleting all commas and brackets. Namely, if $\textsf{Length}(B)=m$, for some integer $\mu>0$, then $\textsf{Seq}(B)$ will have the form $X_\mu X_{\mu+1}\hdots X_{\mu+m}$ . Notice that this implies a restriction on the possible choices for the letter sequences of the sub-brackets $B_1$ and $B_2$, since if $B=[B_1,B_2]$ and $X_{\mu+m_1}$ is the last letter in the letter sequence $\textsf{Seq}(B_1)$ for some integers $\mu>0$ and $m_1\geq 0$, then the first letter in the letter sequence $\textsf{Seq}(B_2)$, must be $X_{\mu+m_1+1}$. For instance, the brackets $B_1\colonequals [X_3,[X_4,X_5]]$ and $B_2\colonequals [X_5,X_6]$ cannot be the factorization of a formal bracket $B$, for then one would have $B=[B_1,B_2]=[[X_3,[X_4,X_5]],[X_5,X_6]]$, i.e., $B$ is not a formal bracket, according to our convention.}
\end{enumerate}

\subsection{Basic Sub-Brackets and Diff-Degree}
We begin with the following definition:
\begin{definition}\label{bbdd}
Let $B$ be a formal bracket. We say that a sub-bracket $S$ of $B$ is a \emph{basic sub}-\emph{bracket}, if it satisfies the following conditions:
\begin{enumerate}[(i)]
\item{$\emph{\textsf{\LL}}(S)\leq 2$;}
\item{If $\emph{\textsf{\LL}}(S)=1$, i.e., $S=X_j$ for some $j \in \nn$, then neither $[X_{j-1},X_j]$ nor $[X_j,X_{j+1}]$ is a sub-bracket of $B$.}
\end{enumerate}
The \emph{diff}-\emph{degree} of a {formal} bracket $B${, denoted by $\dime(B)$,} is the number of basic sub-brackets of $B$.
\end{definition}

\begin{remark}
The use of the terminology ``diff-degree of a formal bracket $B$'' in Definition~\ref{bbdd}, refers to the fact that it coincides with the largest number of differentiation {operations} showing up when the variables of $B$ are replaced with vector fields. The diff-degree plays a crucial role in determining the limiting procedure needed for defining a set-valued iterated Lie bracket (see {Definition~\ref{set-bf}}).
\end{remark}

For example, if $B\colonequals [[X_2,X_3],X_4]$ is a formal bracket, then the sub-brackets $[X_2,X_3]$ and $X_4$ are basic {sub-brackets} of {$B$}, while {the sub-brackets} $X_2$ and $X_3$ are not basic, {which implies that $\dime(B)=2$}. For any formal bracket $B$, one has
\begin{equation*}
\dime(B)=1~\text{if and only if}~\textsf{\LL}(B)\leq 2
\end{equation*}
and furthermore, for every {formal} bracket $B$ with $\textsf{\LL}(B)\geq 2$, one has
\begin{equation*}
\dime(B)\leq\textsf{\LL}(B)-1=\text{number of commas in $B$}.
\end{equation*}
For instance, the formal brackets $B_{1}\colonequals [[X_3,X_4],[X_5,X_6]]$ and $B_{2}\colonequals [[[X_3,X_4],X_5],X_6]$, verify $\textsf{\LL}({B_{1}})=\textsf{\LL}({B_{2}})=4$, $2=\dime({B_{1}})=\textsf{\LL}({B_{1}})-2$ and $3=\dime({B_{2}})=\textsf{\LL}({B_{2}})-1$. {For any formal bracket $B$ with $\textsf{\LL}(B)>2$,} the diff-degree is additive with respect to the factorization of $B$, i.e., if $B=[B_1,B_2]$, then
\begin{equation*}\label{add-form}
\dime(B)=\dime(B_1)+\dime(B_2).
\end{equation*}

If $B$ is a {formal} bracket and {$S$ is a sub-bracket of $B$}, let us use the notation $\dime(S;B)$ to denote {the number of required differentiation operations of $S$ in $B$ (once the variables in $B$ are replaced with vector fields). More precisely, we have the following definition:
\begin{definition}
If $B$ is a formal bracket and $S$ is a sub-bracket of $B$, then:
\begin{enumerate}[(i)]
\item{{First, we set $\dime(B;B)\colonequals 0$;}}
\item{Then, we set $\dime(S_1;B)\colonequals\dime(S_2;B)\colonequals\dime([S_1,S_2];B)+1$.}
\end{enumerate}
\end{definition}

For example, if $B\colonequals [X_3,[X_4,X_5]]$ is a formal bracket, then $\dime(X_3;B)=1$, $\dime(X_4;B)=2$, $\dime(X_5;B)$ $=2$ and $\dime([X_4,X_5];B)=1$. It is not too hard to verify that $\dime(S;B)$ is equal to the number of right brackets that occur in the formal bracket $B$ to the right of the sub-bracket $S$ of $B$ minus the number of left brackets that occur in $B$ to the right of $S$. {It is also clear that} if $(B_1,B_2)$ is the factorization of {a formal bracket} $B$, with $\textsf{Seq}(B_1)=X_{\mu+1}\hdots X_{\mu+m_1}$ and {$\textsf{Seq}(B_2)=X_{\mu+m_1+1}\hdots X_{\mu+m_1+m_2}$} for some {integers} $\mu\geq 0$ {and} {$m_1,m_2\geq 1$}, {then}
\begin{equation*}
\dime(X_j;B)=\begin{cases}
\dime(X_j;B_1)+1, &\text{if $j\in\{\mu+1,\hdots,\mu+m_1\}$},\\
\dime(X_{j};B_2)+1, &\text{if $j\in\{\mu+m_1+1,\hdots,\mu+m_1+m_2\}$},
\end{cases}
\end{equation*}
and moreover, if $X_j$ is a {variable} of a sub-bracket $S$ of $B$ for some $j\in\{\mu+1,\hdots,\mu+m_1+m_2\}$, then
\begin{equation*}
\dime(X_j;B)=\dime(X_j;S)+\dime(S;B).
\end{equation*}
For a formal bracket $B$, the relationship between the numbers $\dime(B)$ and $\dime(X;B)$, such that $X$ is a variable of $B$, is given by the following result:
\begin{proposition}
For any {formal} bracket $B$, {the following equality holds:}
\begin{equation*}
\dime(B)=\max\left\{\dime(X;B)\; : \;\emph{$X$ is a variable of $B$}\right\}.
\end{equation*}
\end{proposition}

\subsection{Bracket-Regularity for Tuples of Vector Fields}\label{familyclass}
Let $\mu,m,r$ be integers such that $\mu\geq 0$, $m\geq 1$ and $r\ge \mu+m$. {For} any {formal} bracket $B$ with {$\textsf{Seq}(B)=X_{\mu+1}\hdots X_{\mu+m}$} and for any {$r$-tuple ${\bf g}=(g_1,\hdots,g_{r})$ of vector fields}, we use the notation $B({\bf g})$ {to denote the vector field that one obtains}, when each variable $X_j$, $j=\mu+1,\hdots,\mu+m$, that appears in $B$, is replaced with the corresponding vector field $g_j$, $j=\mu+1,\hdots,\mu+m$. To make this precise, let us begin with the following definition:

\begin{definition}[Regularity of vector fields]
Let {$k\geq 0$ be an integer}. We say {that} a vector field $g$ on a finite-dimensional manifold $\mM$ {of class $C^{k+1}$} (respectively, $C^{k+2}$), is of class $C^k$ (respectively, $C^{k,1}$), and write $g\in C^{k}$ (respectively, $g\in C^{k,1}$), if it is {$k$-times} continuously differentiable and its $k$-th order derivative is continuous (respectively, locally Lipschitz continuous).\footnote{We write ``$C^{k,1}$'' for the sake of notational convenience, although a more correct notation to use would be ``$C_{\text{loc}}^{k,1}$''. Also, as is customary, the 0-$th$ order derivative of a map $g$, is defined to be $g$ itself.}
\end{definition}

In what follows, we {adopt} the following convention:
\begin{enumerate}[$\bullet$]
\item When considering a manifold $\mM$, together with the formal brackets $B_1,\hdots,B_\ell$, we always mean that $\mM$ is a finite-dimensional manifold of class $C^q$, for some integer $q\geq\max\left\{1,\dime(B_j)\; : \;j\in\{1,\hdots,\ell\}\right\}$.
\end{enumerate}

The following definition is also required in what is to follow shortly:
\begin{definition}[Bracket-Regularity for tuples of vector fields]\label{asindef}
Let $B$ be a {formal} bracket{, with $\emph{\textsf{Seq}}(B)=X_{\mu+1}\hdots X_{\mu+m}$ for some integers $\mu\geq 0$ and {$m\geq 1$}}. {Moreover, let} ${\bf g}=(g_1,\hdots,g_{r})$ be an {$r$-tuple} of vector fields on a manifold $\mM$, with $r\geq \mu+m$ and let $k\geq 0$ be an integer.
\begin{enumerate}[(i)]
\item{When $\emph{\textsf{\LL}}(B)=1$, i.e., $B=X_{\mu+1}$ , we say that ${\bf g}$ is of class $C^{B+k}$ (respectively, $C^{B+k-1,1}$, provided that $k\geq 1$), if the vector field $g_{\mu +1}$, is of class $C^{k}$ (respectively, $C^{k-1,1}$);}
\item{When $\emph{\textsf{\LL}}(B)>1$, we say that ${\bf g}$ is of class $C^{B+k}$ (respectively, $C^{B+k-1,1}$), if the vector field $g_j$ is of class $C^{\dime(X_j;B)+k}$ (respectively, $C^{\dime(X_j;B)+k-1,1}$) for every $j\in\{\mu+1,\hdots,\mu+m\}$.}
\end{enumerate}
We also sometimes write ${\bf g}\in C^{B+k}$ (respectively, ${\bf g}\in C^{B+k-1,1}$) to indicate that ${\bf g}$ is of class $C^{B+k}$ (respectively, $C^{B+k-1,1}$).
\end{definition}

For {instance}, if $B\colonequals [[X_3,X_4],[[X_5,X_6],X_7]]$ is a formal bracket and ${\bf g}=(g_1,\hdots,g_8)$ is a $8$-tuple of vector fields, then ${\bf g}\in C^{B+3}$ (respectively, ${\bf g}\in C^{B+2,1}$), {if} the vector fields $g_3,g_4,g_7\in C^{5}$ (respectively, $C^{4,1}$) and $g_5,g_6\in C^{6}$ (respectively, $C^{5,1}$). It is not too hard to verify the following result:
\begin{proposition}
{Let the notations be as in Definition~\ref{asindef}, with the additional requirement that $\emph{\textsf{\LL}}(B)>1$. Furthermore, let $(B_1,B_2)$ be the factorization of $B$.} Then ${\bf g}\in C^{B+k}$ (respectively, ${\bf g}\in C^{B+k-1,1}$) if and only if ${\bf g}\in C^{B_1+k+1}\cap C^{B_2+k+1}$ (respectively, ${\bf g}\in C^{B_1+k,1}\cap C^{B_2+k,1}$).
\end{proposition}

In the following definition, the idea of ``replacing variables with vector fields'' is made precise:
\begin{definition}\label{plug}\footnote{{Although} this definition uses local coordinates, it turns out to be chart-independent.}
For {given} integers {$r,\mu\ge 0$ and $m\geq 1$} such that {$r\geq\mu+m$}, let $B$ be a formal bracket such that {$\emph{\textsf{Seq}}(B)=X_{\mu+1}\hdots X_{\mu+m}$}, and let ${\bf g}=(g_1,\hdots,g_r)$ be an $r$-tuple of {vector} fields of class $C^{B-1,1}$ (respectively, $C^{B}$, if $\emph{\textsf{\LL}}(B)=1$) on a manifold $\mM$. We define the vector field $B(\mathbf{g})$, in a recursive manner as follows:
\begin{enumerate}[(i)]
\item{First, for any sub-bracket $S$ of $B$, with $\emph{\textsf{\LL}}(S)=1$, i.e., $S=X_j$ for some $j\in\{\mu+1,\hdots,\mu+m\}$, {we} set
\begin{equation*}
S({\bf g})(x)\colonequals g_j(x)~\text{for every}~x\in\mM.
\end{equation*}
If $\emph{\textsf{\LL}}(B)=1$, then we are done, otherwise we proceed to item $(ii)$ given below;}
\item{Next, for any sub-bracket $S$ (with $(S_1,S_2)$ being the factorization of $S$) of $B$, with $\emph{\textsf{\LL}}(S)=2$, {we} set
\begin{equation*}
S({\bf g})(x)\colonequals DS_2({\bf g})(x)\cdot S_1({\bf g})(x)-DS_1({\bf g})(x)\cdot S_2({\bf g})(x)~\begin{cases}
\text{for every $x\in\mM$, in {the} case {when} $S\neq B${,}}\\
\text{for almost every $x\in\mM$, in {the} case {when} $S=B$.}
\end{cases}
\end{equation*}
If $\emph{\textsf{\LL}}(B)=2$, then we are done, otherwise we proceed to item $(iii)$ given below;}
\item{Repeat the procedure outlined in item $(ii)$ given above, next for sub-brackets $S$ of $B$, with $\emph{\textsf{\LL}}(S)=3$ and so on in a strictly increasing order, until $\emph{\textsf{\LL}}(S)=\emph{\textsf{\LL}}(B)$.}
\end{enumerate}
\end{definition}

\begin{remark}
Note that Definition~\ref{plug} implies that the vector field $B({\bf g})$ is of class $L^\infty_{\text{loc}}$ (in fact, it is of class $C^{0}$, when $\text{\textsf{\LL}}(B)=1$).
\end{remark}

\subsection{Set-Valued Iterated Lie brackets}\label{setB}
We define set-valued iterated Lie brackets in $\rr^m$ in a chart-dependent fashion. However, the definition turns out to be chart-independent (see \cite{Feleqi2017}), so that it is meaningful on manifolds as well.
\begin{definition}[Set-Valued iterated Lie bracket]\label{set-bf}
For given integers {$r,\mu\geq 0$ and $m\ge 1$} such that {$r\geq\mu+m$}, let $B$ be a formal bracket of $\emph{\textsf{\LL}}(B)=m$ and let ${\bf g}\colonequals(g_1,\hdots,g_r)$ be an $r$-tuple of vector fields of class $C^{B-1,1}$ (respectively, $C^{B}$, if $\emph{\textsf{\LL}}(B)=1$) on $\rr^m$, for some integer $n\ge1$. Let $S_{1},\hdots,S_{d}$, where $d\colonequals\dime(B)$, be the basic sub-brackets of $B$ (which have been ordered lexicographically) and consider a $d$-tuple $\h=(h_1,\hdots,h_d)\in (\rr^m)^d$ of vectors in $\rr^m$. {If $\emph{\textsf{\LL}}(B)>1$, let ${\bf g}^\h\colonequals(g_{\mu+1}^\h,\hdots,g_{\mu+m}^\h)$ be an $m$-tuple of vector fields on $\rr^m$, where for each $j\in\{\mu+1,\hdots,\mu+m\}$, the vector field $g_j^\h$ is defined as follows:}
\begin{equation*}
g_j^\h(x)\colonequals g_j(x+h_i)~\text{for every}~x\in\rr^m,
\end{equation*}
where $i\in\{1,\hdots,d\}$ is the (unique) integer such that the variable $X_j\in S_i$. We now define the \emph{set}-\emph{valued iterated Lie bracket} $B_{\emph{set}}({\bf g})$ as follows:
\begin{enumerate}[(i)]
\item{If $\emph{\textsf{\LL}}(B)=1$, i.e., $B=X_j$ for $j=\mu+1$, then we set
\begin{equation*}
B_{\emph{set}}({\bf g})(x)\colonequals B({\bf g})(x)=g_j(x)~\text{for every}~x\in\rr^m;
\end{equation*}}
\item{If $\emph{\textsf{\LL}}(B)>1$, then we set
\begin{equation*}
B_{\emph{set}}({\bf g})(x)\colonequals{\emph{\bf co}}\left\{v\; : \;v=\lim_{k\to\infty}B\left({\bf g}^{\h_k}\right)(x)\right\}~\text{for every}~x\in\rr^m,
\end{equation*}
where the limit is taken along all sequences $\{\h_k\}_{k\in\nn}=\{(h_{1k},\hdots,h_{dk})\}_{k\in\nn}\subset{(\rr^m)}^d$ converging to ${\bf 0}$ such that $\{(x+h_{1k},\hdots,x+h_{dk})\}_{k\in\nn}\subset\emph{\textsf{Diff}}^{\dime(S_{1};B)}(S_{1}({\bf g}))\times\cdots\times\emph{\textsf{Diff}}^{\dime(S_{d};B)}(S_{d}({\bf g}))$.\footnote{The notation $\textsf{Diff}^{k}(g)$, for an integer $k>1$ denotes the set consisting of the points of $k$-$th$ order differentiability of a map $g$.}}
\end{enumerate}
\end{definition}

\begin{remark}
At first sight, Definition~\ref{set-bf} might appear a bit involved. However, it becomes quite intuitive as soon as one considers some special cases of it, as illustrated by the examples given below.
\end{remark}

\begin{examples}\label{examples}
\begin{enumerate}[(i)]
\item{If $B\colonequals [[X_1,X_2],X_3]$ is a formal bracket and ${\bf g}=(g_1,g_2,g_3,g_4,g_5)$ is a 5-tuple of vector fields of class $C^{B-1,1}$, then for every $x\in\rr^m$, the set-valued iterated Lie bracket $B_{\text{set}}({\bf g})(x)$ is defined as follows:
\begin{align*}
&B_{\text{set}}({\bf g})(x)=[[g_1,g_2],g_3]_{\text{set}}(x)\colonequals\\
&\text{\bf co}\left\{v\in\rr^m\; : \;v=\lim_{k\to\infty}\Big(Dg_3(x+h_{1k})\cdot [g_1,g_2](x+h_{2k})-D[g_1,g_2](x+h_{2k})\cdot g_3(x+h_{1k})\Big)\right\},
\end{align*}
where the limit is taken along all sequences $\{(h_{1k},h_{2k})\}_{k\in\nn}$ converging to ${\bf 0}$ such that $\{(x+h_{1k},x+h_{2k})\}_{k\in\nn}\subset\text{\textsf{Diff}}^{1}(g_3)\times\text{\textsf{Diff}}^{1}([g_1,g_2])$.}
\item{Let $B\colonequals [[X_1,X_2],[X_3,X_4]]$ be a formal bracket and let the 6-tuple ${\bf g}=(g_1,g_2,g_3,g_4,g_5,g_6)$ of vector fields be of class $C^{B-1,1}$, then for every $x\in\rr^m$, the set-valued iterated Lie bracket $B_{\text{set}}({\bf g})(x)$ is defined as follows:
\begin{align*}
&B_{\text{set}}({\bf g})(x)=[[g_1,g_2],[g_3,g_4]]_{\text{set}}(x)\colonequals\\
&\text{\bf co}\left\{v\in\rr^m\; : \;v=\lim_{{k\to\infty}}\Big(D[g_3,g_4](x+h_{1k})\cdot [g_1,g_2](x+h_{2k})-D[g_1,g_2](x+h_{2k})\cdot[g_3,g_4](x+h_{1k})\Big)\right\},
\end{align*}
where the limit is taken along all sequences $\{(h_{1k},h_{2k})\}_{k\in\nn}$ converging to ${\bf 0}$ such that $\{(x+h_{1k},x+h_{2k})\}_{k\in\nn}\subset\text{\textsf{Diff}}^{1}([g_3,g_4])\times\text{\textsf{Diff}}^{1}([g_1,g_2])$.}
\item{If $B\colonequals [[[X_2,X_3],X_4],X_5]$ is a formal bracket and ${\bf g}=(g_1,g_2,g_3,g_4,g_5,g_6)$ is a 6-tuple of vector fields of class $C^{B-1,1}$, then for every $x\in\rr^m$, the set-valued iterated Lie bracket $B_{\text{set}}({\bf g})(x)$ is defined as follows:
\begin{align*}
&B_{\text{set}}({\bf g})(x)=[[[g_2,g_3],g_4],g_5]_{\text{set}}(x)\colonequals\\
&\text{\bf co}\left\{v\in\rr^m\; : \;v=\lim_{{k\to\infty}}\Big(Dg_5(x+h_{1k})\cdot\big(Dg_4(x+h_{2k})\cdot [g_2,g_3](x+h_{3k})-D[g_2,g_3](x+h_{3k})\cdot g_4(x+h_{2k})\big)\right.\\
&\left.\qquad-D\big(Dg_4(x+h_{2k})\cdot [g_2,g_3](x+h_{3k})-D[g_2,g_3](x+h_{3k})\cdot g_4(x+h_{2k})\big)\cdot g_5(x+h_{1k})\Big)\right\},
\end{align*}
where the limit is taken along all sequences $\{(h_{1k},h_{2k},h_{3k})\}_{k\in\nn}$ converging to ${\bf 0}$ such that $\{(x+h_{1k},x+h_{2k},x+h_{3k})\}_{k\in\nn}\subset\text{\textsf{Diff}}^{1}(g_5)\times\text{\textsf{Diff}}^{2}(g_4)\times\text{\textsf{Diff}}^{2}([g_2,g_3])$.}
\end{enumerate}
\end{examples}

In the following result, we collect some {elementary} properties {of} set-valued Lie brackets:
\begin{proposition}\label{prop-svb}
Let the notations be as in Definition~\ref{set-bf}. {Then the following statements hold}:
\begin{enumerate}[(i)]
\item{\emph{(Upper semi-continuity)} The set-valued {mapping} {$x\mapsto B_{\emph{set}}({\bf g})(x)$} is \emph{upper semi}-\emph{continuous}, and {takes non-empty compact and convex values};\footnote{{A set-valued map $\Phi\colon X\rightrightarrows Y$ between topological spaces $X$ and $Y$ is said to be upper semi-continuous at a point $x\in X$, if for every neighborhood $\mathcal{O}_2$ of the subset $\Phi(x)\subset Y$, there exists a neighborhood $\mathcal{O}_1$ of $x$ such that $\Phi(\mathcal{O}_1)\subset\mathcal{O}_2$.}}}
\item{\emph{(Consistency with the classical case)} If ${\bf g}$ is of class $C^B$ in a neighborhood $\mathcal{O}$ of a point $\bar x\in\rr^m$, then the (single-valued) mapping $\mathcal{O}\ni x\mapsto B({\bf g})(x)$ is continuous and
\begin{equation*}
B_{\emph{set}}({\bf g})(x)=\{B({\bf g})(x)\}~\text{for every}~x\in \mathcal{O};
\end{equation*}
\item\emph{(Anti-Symmetry)} If $B\colonequals [X_{\mu+1},X_{\mu+2}]$, ${\bf\hat g}=(g_{\mu+1},g_{\mu+2})$ and ${\bf\check g}=(g_{\mu+2},g_{\mu+1})$, then
\begin{equation*}
B_{\emph{set}}({\bf\hat g})(x)=[g_{\mu+1},g_{\mu+2}]_{\emph{set}}(x)=-[g_{\mu+2},g_{\mu+1}]_{\emph{set}}(x)=-B_{\emph{set}}({\bf\check g})(x)
~\text{for every}~x\in\mM;\footnote{For a given set $W\subset\rr^{n}$, we set $-W\colonequals\{-w\; : \;w\in W\}$. {Also, observe that for a vector field $g$ of class $C^{0,1}$ in a neighborhood $\mathcal{O}$ of a point $\bar x\in\rr^m$, one has that $[g,g]_{\text{set}}(x)=\{\bf 0\}$ for every $x\in\mathcal{O}$.}}
\end{equation*}}
\item{\emph{(Consistency with Clarke's generalized Jacobian)} If $(B_1,B_2)$ is the factorization of $B$ and ${\bf g}$ is of class $C^{B_1,1}\cap C^{B_2+1}$ in a neighborhood $\mathcal{O}$ of a point $\bar x\in\rr^m$, then for every $x\in \mathcal{O}$, one has that
\begin{align*}
B_{\emph{set}}({\bf g})(x)&=DB_2({\bf g})(x)\cdot B_1({\bf g})-D_CB_1({\bf g})(x)\cdot B_2({\bf g})(x)\\
&=\{v\; : \;v=DB_2({\bf g})(x)\cdot B_1({\bf g})-L\cdot B_2({\bf g})(x),~\emph{where}~L\in D_CB_1({\bf g})(x)\},
\end{align*}
with the notation {$D_CB_1({\bf g})(x)$} denoting the Clarke's generalized Jacobian {of $B_1({\bf g})$ at $x$}.\footnote{The \emph{Clarke's generalized Jacobian} $D_C\varphi(x)$ of a $C^{0,1}$ function $\varphi\colon\rr^n\to\rr^m$ at a point $x\in\rr^n$, is {defined as follows:}
\begin{equation*}
D_C\varphi(x)\colonequals{\text{\bf co}}\left\{A\; : \;A=\lim_{j\to\infty} D\varphi(x_j),~\text{where}~\{x_{j}\}_{j\in\nn}\subset\text{\textsf{Diff}}^{1}(\varphi)~\text{is a sequence of points in $\rr^n$, with}~x_j\to x~\text{as}~j\to\infty\right\}.
\end{equation*}}}
\end{enumerate}
\end{proposition}

In Proposition~\ref{prop-svb}, the proofs for items $(ii)$ and $(iv)$ are fairly straightforward, while we refer the reader to \cite[Proposition~3.2]{Feleqi2017} and \cite[Proposition~3.3]{Rampazzo2007commutators} for the proofs of items $(i)$ and $(iii)$, respectively.

\begin{remark} As for a possible relation with Clarke's generalized Jacobian, let us 
notice that, if the $r$-tuple of vector fields ${\bf g}$ were of class $C^{B_1,1}\cap C^{B_2,1}$ in a neighborhood $\mathcal{O}$ of a point {$\bar x\in\rr^m$}, then one might be tempted to define (in analogy with the formula for the Lie bracket in the classical case) a different set-valued Lie bracket $\hat B_{\text{set}}({\bf g})$, by setting
\begin{equation*}
\hat B_{\text{set}}({\bf g})(x)\colonequals\{v\; : \;v=L_2\cdot B_1({\bf g})-L_1\cdot B_2({\bf g})(x),~\text{where}~(L_1,L_2)\in D_CB_1({\bf g})(x)\times D_CB_2({\bf g})(x)\}
\end{equation*}
for every $x\in\mathcal{O}$. However, such a definition of a set-valued Lie bracket would not be a good choice, since, for example, it would result in a Lie bracket that is ``too large'' for the associated ``commutativity property'' to hold (see \cite{Rampazzo2007commutators} for more details).
\end{remark}

\section{Generalized Differential Quotients and an Estimate for Obtaining an Asymptotic Formula}
\subsection{Generalized Differential Quotients}\label{s-GDQ}
The theory of GDQs is a type of set-valued generalization of the theory of classical differentials, which has been created primarily to tackle the case of set-valued maps in quite a relaxed general setting. This theory has been initiated and further developed by H. J. Sussmann\footnote{Sussmann has stated to have been inspired by the following characterization of classic differentiability from a paper of Botsko and Gosser (see \cite{Botsko1985}): a function $F\colon U \to \rr^m$, where $U$ is an open set in $\rr^n$, is differentiable at some point $x_* \in U$ if and only if there exists a map $G \colon U \to \textsf{Lin}(\rr^n,\rr^m)$, which is continuous at $x_*$ and such that $F(x) = F(x_*) + G(x)(x-x_*)$ for every $x \in U$. Here $\textsf{Lin}(\rr^n,\rr^m)$ stands for the space of linear maps from $\rr^n$ to $\rr^m$. 
}, and we refer the interested reader to \cite{Sussmann2000new,Sussmann2000resultats,Sussmann2008} for the proofs of the results given in this section, with the exception of the proofs of Theorem~\ref{chain}, Theorem~\ref{open-t}, and Theorem~\ref{sucole}, which are given in this section, Appendix~\ref{a-OMT} and Appendix~\ref{a-GDQ}, respectively. We show that the set-valued Lie brackets are GDQs of suitable compositions of the flows generated by the vector fields $\{f_1,\hdots,f_\nu\}$ in \eqref{sistemaintro1}, where some of these might be merely continuous, so that their corresponding flows can possibly be multi-valued. The chain rule (see Theorem~\ref{chain}) and the open mapping theorem (see Theorem~\ref{open-t}) valid within the framework of the theory of GDQs, play a crucial role in the proof of the main results, namely Theorem~\ref{ChowRas} and Theorem~\ref{wChowRas}. Let us begin with a few definitions:
\begin{definition}\label{gph-def}
Given a set-valued map $F\colon X\rightrightarrows Y$ (with possibly empty values) from a metric space $X$ into a metric space $Y$, the set
\begin{equation*}
\emph{\textsf{Gr}}(F)\colonequals\{(x,y)\in X\times Y\; : \;y\in F(x)\},
\end{equation*}
is called the \emph{graph} of $F$.
\end{definition}

\begin{definition}\label{igc-def}
Let $X$ and $Y$ be metric spaces. Assuming that the graph $\emph{\textsf{Gr}}(F)$ is compact, we say that a sequence {$\{F_j\}_{j\in\nn}$} of set-valued {maps} {with compact graphs} \emph{inward graph converges} to $F$, if for any {non-empty open set $\O\subset X\times Y$} such that {$\emph{\textsf{Gr}}(F)\subset \O$,} there exists {a} $j_{\O}\in\nn$ such that {$\emph{\textsf{Gr}}(F_{j})\subset \O$} {for every $j\ge j_{\O}$}.
\end{definition}

\begin{definition}\label{reg-def}
{Let $X$ and $Y$ be metric spaces}. We say that a set-valued map $F\colon X\rightrightarrows Y$ is \emph{Cellina continuously approximable (CCA)}, if for every {non-empty compact set $K\subset X$}{,} the restriction $F|_K$ {of $F$ to $K$}, has {a} compact graph and {it} is a limit (in the sense of Definition \ref{igc-def}) of a sequence of continuous single-valued maps from $K$ to $Y$.
\end{definition}

\begin{remark}
Definition~\ref{reg-def} is a generalization of the concept of continuity. Indeed, it is not too hard to verify that if $F\colon X\to Y$ is a single-valued map from a metric space $X$ to a metric space $Y$, then $F$ is {CCA} if and only if $F$ is continuous.
\end{remark}

The following result (see \cite[Theorem~3.6]{Sussmann2008}) shows that composition preserves {CCA set-valued} maps:
\begin{theorem}\label{CCA-comp}
If {the set-valued maps} $F\colon X\rightrightarrows Y$ and $G\colon Y\rightrightarrows Z$ are {CCA}, then the composite map $G\circ F\colon X\rightrightarrows Z$ is {CCA}.
\end{theorem}

{Before proceeding further, let us fix another notation. For real linear spaces $X$ and $Y$, let us use the notation $\textsf{Lin}(X,Y)$ to denote the set of all linear maps from $X$ to $Y$. A \emph{linear multi}-\emph{map} from $X$ to $Y$, is a {subset} of $\textsf{Lin}(X,Y)$.}
\begin{definition}[Generalized differential quotient]\label{GDQ-def}
Let ${F}\colon \rr^n\rightrightarrows \rr^m$ {be a set-valued map} and let {$\L\subset\emph{\textsf{Lin}}(\rr^n,\rr^m)$} be a {non-empty} compact set. {{Moreover}, let} {$S\subset\rr^n$ be a non-empty set}. We say that $\L$ is a \emph{GDQ of $F$ at $(x_*,y_*)\in\rr^n\times\rr^m$}, \emph{with $y_*\in F(x_*)$}, \emph{in the direction of $S$}, if for every real number {$\d>0$,} there exist $U$ {and} $G$ such that:
\begin{enumerate}[(i)]
\item{$U$ is a {compact} neighborhood of {${\bf 0}$} in $\rr^n$ and $(x_*+U)\cap S$ is compact;}
\item{$G$ is a {CCA} set-valued map from $(x_*+U)\cap S$ to the {$\d$-neighborhood} {$\L^{\d}$} of $\L$;\footnote{Here, the set $\L^{\d}\colonequals\{A\; : \;\inf_{L\in \L}|A-L|\leq\d\}$, where $|\cdot|$ denotes the operator norm.}}
\item{{$y_*+G(x)(x-x_*)\subset F(x)$} for every {$x\in (x_*+U)\cap S$}.\footnote{Here, the set $y_*+G(x)(x-x_*)\colonequals\{y\; : \;y=y_*+L(x-x_*),~\text{where}~L\in G(x)\}$.}}
\end{enumerate}
\end{definition}

\begin{remark}
{An extension of Definition~\ref{GDQ-def} to the setting of finite-dimensional manifolds is also possible. More precisely, {let $\mM$ and $\mN$ be finite-dimensional manifolds} {of class $C^{1}$}. {Moreover}, let $F\colon\mM\rightrightarrows\mN$ be a set-valued map and let {$S\subset\mM$} be a non-empty set. Let $x_*\in\mM$ and $y_*\in F(x_*)$, and let us choose coordinate charts $(U,\phi)$ and $(V,\psi)$, centered at $x_*$ and $y_*$, respectively. We now declare a set $\L\subset\text{\textsf{Lin}}(T_{x_*}\mM,T_{y_*}\mN)$ to be a GDQ of $F$ at $(x_*,y_*)$ in the direction of $S$, if {$D\psi(y_*)\circ\L\circ{D\phi(x_*)}^{-1}$} is a GDQ of $\psi\circ F\circ\phi^{-1}$ at $({\bf 0},{\bf 0})$ in the direction of $\phi(U\cap S)$. It turns out that this definition is intrinsic, i.e., it does not depend on the choice of the coordinate charts $(U,\phi)$ and $(V,\psi)$.}
\end{remark}

As one might expect, the notions of the classical differential and also Clarke's generalized Jacobian are particular instances of the notion of a GDQ. More precisely, the following two lemmas hold:
\begin{lemma} Let $\mM$ and $\mN$ be finite-dimensional manifolds {of class $C^{1}$} and let $F\colon\mM\to\mN$ be a map, {classically differentiable} at a point $x\in\mM$. Then, the singleton $\{DF(x)\}\subset\emph{\textsf{Lin}}(T_x\mM,T_{F(x)}\mN)$, is a GDQ of $F$ at $(x,F(x))$ in the direction of $\mM$.
\end{lemma}

\begin{lemma}
Let $\mM$ and $\mN$ be finite-dimensional manifolds {of class $C^{1}$} and let $F\colon\mM\to\mN$ be a Lipschitz continuous map in a neighborhood of a point $x\in\mM$. Then, {the Clarke's generalized Jacobian $D_CF(x)\subset\emph{\textsf{Lin}}(T_x\mM,T_{F(x)}\mN)$}, is a GDQ of $F$ at $(x,F(x))$ in the direction of $\mM$.
\end{lemma}

\begin{remark}
The notion of a GDQ is in fact a strict generalization of both the notions of the classical differential and Clarke's generalized Jacobian. For instance (see \cite{Sussmann2002}), consider the following function:
\begin{equation*}
g(x)\colonequals\begin{cases}
x\sin\left(\dfrac{1}{x}\right), &\text{if $x\in\rr\setminus\{0\}$,}\\
0, &\text{if $x=0$}.
\end{cases}
\end{equation*}
It is not too hard to verify that the set $[-1,1]\subset\rr$ is a GDQ of the function $g$ at $(0,0)$ in the direction of $\rr$, whereas both the classical differential and Clarke's generalized Jacobian of $g$ at the point $0\in\rr$ do not exist.
\end{remark}

The following result establishes the chain rule property of GDQs:
\begin{theorem}[Chain rule]\label{chain}
Let $\mM_1,\mM_2$ and $\mM_3$ be finite-dimensional manifolds of class $C^1$, $x_1\in\mM_1$, $x_2\in\mM_2$ and $x_3\in\mM_3$, $F_1\colon\mM_1\rightrightarrows\mM_2$ and $F_2\colon\mM_2\rightrightarrows\mM_3$ be set-valued maps such that $x_2\in F_1(x_1)$ and $x_3\in F_2(x_2)$, and let $S_1\subset\mM_1$ and $S_2\subset \mM_2$ be non-empty sets such that $F(S_1) \subset S_2$. Assume that $\Lambda_1\subset\emph{\textsf{Lin}}(T_{x_1}\mM_1, T_{x_2}\mM_2)$ is a GDQ of $F_1$ at $(x_1,x_2)$ in the direction of $S_1$ and $\Lambda_2\subset\emph{\textsf{Lin}}(T_{x_2}\mM_2, T_{x_3}\mM_3)$ is a GDQ of $F_2$ at $(x_2, x_3)$ in the direction of $S_2$. 
If {$U\cap S_2$} is a \emph{retract}\footnote{If $X$ is a topological space and if $W$ is a subspace of $X$, then a continuous map $\textsf{r}\colon X\to W$ is called a retraction, if the restriction $\textsf{r}|_W$ of $\textsf{r}$ to $W$, is the identity map on $W$, i.e., $\textsf{r}(w)=w$ for all $w\in W$. In this case, $W$ is called a retract of X.} of $U$ for some compact neighborhood $U$ of $x_2$ in $\mM_2$ or if $F_1$ is {a single-valued map}, then the set $\L_2\circ\L_1\colonequals\{L_2\circ L_1\; :\; (L_1,L_2)\in\L_1\times\L_2\}$ is a GDQ of $F_2\circ F_1$ at $(x_1,x_3)$ in the direction of $S_1$.
\end{theorem}

\begin{proof}
We only prove the ``non-directional'' case, where $S_1 = \mM_1$ and $S_2 = \mM_2$, which suffices for our purposes. In this case, the ``retraction hypothesis'' holds trivially, as $U \cap S_2 = U \cap \mM_2 = U$ is clearly a retract of $U$. The proof of the ``directional variant'' requires only slight further modifications. We provide the proof here because the original proof given in \cite{Sussmann2008} is for ``approximate GDQs'' (which is another generalization of the theory of classical differentials), with modifications for the case of GDQs being hinted only at the end, spans over four pages, whereas our version is based on unpublished notes of H. J. Sussmann.

Without loss of generality, we can assume that $\mM_i = \rr^{n_i}$ for some integers $n_i>0$ and $x_i = {\bf 0}$, $i=1,2,3$. Fix $\d >0$. Let $\L\colonequals \L_2 \circ \L_1$ and choose $\delta_1, \delta_2> 0$ such that
\begin{equation*}
\L_2^{\delta_2} \circ \L_1^{\delta_1}
\subset \L^\delta.
\end{equation*}
According to Definition~\ref{GDQ-def}, let $G_1 \colon U_1 \subset \rr^{n_1} \rightrightarrows \L_1^{\delta_1}$ be a CCA set-valued map defined on a compact neighborhood $U_1$ of ${\bf 0}$ in $\rr^{n_1}$ such that $G_1(x)x \subset F_1(x)$ for every $x\in U_1$. Similarly, let $G_2 \colon U_2 \subset \rr^{n_2} \rightrightarrows \L_2^{\delta_2}$ be a CCA set-valued map defined on a compact neighborhood $U_2$ of ${\bf 0}$ in $\rr^{n_2}$ such that $G_2(y)y \subset F_2(y)$ for every $y\in U_2$. The neighborhoods $U_1$, $U_2$ can and are chosen as closed balls centered at ${\bf 0}$ in $\rr^{n_1}$ and $\rr^{n_2}$, respectively.

We want to find a CCA set-valued map $G_3 \colon W \subset \rr^{n_1}\rightrightarrows \L_3^\delta$ such that $F_3(x) \supset G_3(x)x$ for every $x \in W$, where $W$ is a compact neighborhood of ${\bf 0}$ in $\rr^{n_1}$. To this end, we can observe that for $x\in U_1$,
\begin{equation*}
F_2 \circ F_1 (x) = F_2(F_1(x)) \supset F_2(G_1(x)x) \supset G_2(G_1(x)x)(G_1(x)x) = G_3(x)x,
\end{equation*}
where
\begin{equation}\label{G3-def}
G_3(x) \colonequals G_2(G_1(x)x) \circ G_1(x),
\end{equation}
provided that the mapping $x \mapsto G_2(G_1(x))$ makes sense, i.e., we need that
\begin{equation}\label{cond-ch-r}
G_1(x)x \subset U_2~\text{for every}~x\in U_1.
\end{equation}
If we take $U_1$ sufficiently small, then inclusion \eqref{cond-ch-r} is satisfied. Indeed, let the radius of the ball $U_2$ be $r_2>0$ and let
\begin{equation*}
\kappa \colonequals \max \{|L| \; : \; L \in \L_1^{\delta_1}\},
\end{equation*}
then for $x \in U_1$ and $L\in G(x)$, since $G_1(x) \subset \L_1^{\delta_1}$, we have that $|Lx| \le \kappa r_1$, where $r_1>0$ is the radius of the ball $U_1$. Then inclusion \eqref{cond-ch-r} holds, provided that $r_1$ is chosen less than or equal to $r_2/\kappa$. Hence, we choose $U_1$ so that inclusion \eqref{cond-ch-r} holds and define $G_3 \colon W \rightrightarrows \L_3^\d$ (with $W \colonequals U_1$) such that for every $x\in W=U_1$, the set $G_3(x)$ is given by \eqref{G3-def}. To be clear, by \eqref{G3-def} we mean
\begin{equation*}
G_3(x)\colonequals\big\{L_2 \circ L_1 \;:\; L_1 \in G_1(x),~L_2 \in G_2(L_1x)\big\}.
\end{equation*}
Clearly $G_3$ is well-defined and $G_3(x)x \subset F_3(x)$ for every $x\in U_1$. To complete the proof, it only remains to show that $G_3$ is CCA. We do this by expressing $G_3$ as a composite of CCA maps as follows:
\begin{equation*}
G_3 = \G_5 \circ \G_4 \circ \G_3 \circ \G_2 \circ \G_1,
\end{equation*}
where
\begin{enumerate}[(a)]
\item $\G_1 \colon U_1 \rightrightarrows U_1 \times \L_1^{\d_1}$, $x\mapsto \{x\} \times G_1(x)$;
\item $\G_2 \colon U_1 \times \L_1^{\d_1} \to U_1 \times \L_1^{\d_1} \times \rr^{n_2}$, $(x, L) \mapsto (x, L, Lx)$;
\item $\G_3 \colon U_1 \times \L_1^{\d_1} \times \rr^{n_2} \to U_1 \times \L_1^{\d_1} \times U_2$, $(x,L,y) \mapsto (x, L, \textsf{r}(y))$, where $\textsf{r} \colon \rr^{n_2} \to U_2$ is a retraction (i.e., a continuous map such that $\textsf{r}(y) = y$ for every $y\in U_2$);
\item $\G_4 \colon U_1 \times \L_1^{\d_1} \times U_2 \rightrightarrows 
U_1 \times \L_1^{\d_1} \times U_2 \times \L_2^{\d_2}$,
$(x, L, y) \mapsto \{x\} \times \{L\} \times \{y \} \times G_2(y)$;
\item $\G_5 \colon U_1 \times \L_1^{\d_1} \times U_2 \times \L_2^{\d_2} \times \L^\d$, 
$(x, L_1, y, L_2) \mapsto L_2 \circ L_1$. 
\end{enumerate}
It is clear that each $\G_i$, $i=1,2,3,4,5$ is CCA, so $G_3$ is CCA in view of Theorem~\ref{CCA-comp} and the proof is complete.
\end{proof}

Before proceeding further, let us fix another notation: for a real number $r>0$ and an integer $n>0$, the notation $\overline{\mathfrak{B}}_r(\bar{x})$ stands for the closed ball in $\rr^n$ of radius $r$ centered at $\bar{x}\in\rr^n$. The following result establishes an open mapping property of GDQs:
\begin{theorem}[Open mapping theorem for set-valued maps via GDQs]\label{open-t}
Let {$F\colon \rr^n\rightrightarrows\rr^m$} be a set-valued map, $x_*\in\rr^n$, $y_*\in F(x_*)$ and let the non-empty compact set $\L\subset\emph{\textsf{Lin}}(\rr^n,\rr^m)$ be a GDQ of $F$ at $(x_*,y_*)$ in the direction of $\rr^n$. If $\L$ consists of surjective linear maps, then $F$ is an \emph{open map} at $(x_*, y_*)$, in the sense that for every neighborhood $U$ of $x_*$ in $\rr^n$, $F(U)$ is a neighborhood of $y_*$ in $\rr^m$. In addition, the ``linear rate property'' holds: there exist real numbers $\bve, \kappa>0$ such that
\begin{equation*}
y_* + \overline{\mathfrak{B}}_\ve({\bf 0}) \subset F\left(x_* + \overline{\mathfrak{B}}_{\kappa\ve}({\bf 0})\right)
\end{equation*}
for every real number $\ve \in (0,\bve]$, i.e., if $y \in \rr^m$ is such that $|y-y_*| \le \ve$, then there exists $x\in \rr^n$ with $|x - x_*| \le \kappa\ve$ such that $y \in F(x)$.
\end{theorem}

Actually, Sussmann has established a more general result in \cite{Sussmann2008}, which also implies a ``directional variant'' of Theorem~\ref{open-t}. However, Theorem~\ref{open-t} suffices for our purposes. For the reader's convenience, we provide a direct and concise proof of Theorem~\ref{open-t} in Appendix~\ref{a-OMT}.

\medskip

In Appendix~\ref{a-GDQ}, we prove the following result, which gives a sufficient condition for a set to be a GDQ:
\begin{theorem}[A sufficient condition for a set to be a generalized differential quotient]\label{sucole}
Let $\mM$ be a finite-dimensional manifold {of class $C^{1}$}, let $q>0$ be an integer and let $F\colon \mM \rightrightarrows \rr^m$ be an upper semi-continuous set-valued map taking non-empty closed values. Let $x_*\in \mM$, $y_*\in F(x_*)$, and let $\L\subset\emph{\textsf{Lin}}(T_{x_*}\mM,\rr^m)$ be a non-empty compact and convex set. Moreover, let $V\subset\mM$ be a neighborhood of $x_*$ (endowed with a Riemannian distance $d\colon V\times V\to\rr$) and for each ${j\in \nn}$, let $F_j\colon V\rightrightarrows\rr^m$ be a set-valued map, such that the following conditions are satisfied:
\begin{enumerate}[(i)]
\item For each $j\in \nn$, $F_j$ is an upper semi-continuous set-valued map taking non-empty compact and convex values;\footnote{For each $j\in\nn$, the set-valued map $F_j$ has a compact graph, in view of the following known fact: if $X,Y$ are metric spaces and the set-valued map $F\colon X\rightrightarrows Y$ takes compact values, then $F$ is upper semi-continuous if and only if $F$ has a compact graph (see \cite{Aubin1984}).}
\item The sequence $\{F_j\}_{j\in \nn}$ of set-valued maps inward graph converges to $F$;
\item For each $j\in\nn$, the following inequality:
\begin{equation*}
\inf\left\{|y-y_*-L(x-x_*)|\; :\; (y,L)\in F_j(x)\times\L\right\}\leq d(x,x_*)\tilde{\o}_j(d(x,x_*))
\end{equation*}
holds for every $x\in V$, where the sequence $\{\tilde{\o}_j\}_{j\in \nn}$ is made of functions which converge uniformly to some modulus $\tilde{\o}$ {on} an interval $[0,r]$, for some real number $r>0$, as $j$ goes to infinity.\footnote{A real-valued function $\omega\colon [0,+\infty)\to [0,+\infty)$, is said to be a \emph{modulus}, if $\omega(0)=0$, and if $\omega$ is increasing and continuous from the right at $0\in\rr$, i.e., $\lim_{r\downarrow 0}\omega(r)=\omega(0)$.}
\end{enumerate}
Then, $\L$ is a {GDQ} of $F$ at $(x_*,y_*)$ in the direction of $\mM$.
\end{theorem}

\begin{corollary}\label{GDQ-c}
Let $\mM$ be a manifold {of class $C^{1}$}, let $q>0$ be an integer and let $F\colon \mM\rightrightarrows\rr^m$ be an upper semi-continuous set-valued map taking non-empty compact and convex values. Let $x_*\in \mM$, $y_*\in F(x_*)$, and let $\L\subset\emph{\textsf{Lin}}(T_{x_*}\mM,\rr^m)$ be a non-empty compact and convex set. Furthermore, let $V\subset\mM$ be a neighborhood of $x_*$ (endowed with a Riemannian distance $d\colon V\times V\to\rr$), such that the following inequality:
\begin{equation*}
\inf\left\{|y-y_*-L(x-x_*)|\; :\; (y,L)\in F(x)\times\L\right\}\leq d(x,x_*)\tilde{\o}(d(x,x_*))
\end{equation*}
holds for every $x\in V$ and for some modulus $\tilde{\o}$. Then, $\L$ is a {GDQ} of $F$ at $(x_*,y_*)$ in the direction of $\mM$.
\end{corollary}

\begin{proof}
{Use Theorem~\ref{sucole}, with the set-valued map $F_j\colonequals F$ and the function $\tilde{\o}_j\colonequals \tilde{\o}$ for each $j\in\nn$.}
\end{proof}

\begin{remark} 
In this paper, we only use Corollary~\ref{GDQ-c}, with the set-valued map $F$ being assumed to be continuous.
\end{remark}

\subsection{Set-Valued Lie Brackets as GDQs of Multi-Flows}
We shortly recall the crucial estimate obtained in \cite{Feleqi2017} (which is valid for a certain composition of the flows of the vector fields {corresponding to the associated set-valued Lie bracket}), using which an asymptotic formula can be established. {Thanks} to this estimate, we are able to show that certain compositions of suitably re-parametrized multi-flow maps possess a GDQ that can be expressed in terms of a set-valued Lie bracket (see Theorem~\ref{GDQ-lem}). To proceed further, we need the following definition:
\begin{definition}[Multi-Flow map]\label{commutator}
Let $B$ be a formal bracket, with $\emph{\textsf{Seq}}(B)=X_{\mu+1}\hdots X_{\mu+m}$ for some integers $\mu\geq 0$ and $m\geq 1$. Moreover, for an integer $r\geq\mu+m$, let $\mathbf{g}=(g_1,\hdots,g_{r})$ be an $r$-tuple of vector fields of class $C^{B-1,1}$ (respectively, $C^{B}$, if $\emph{\textsf{\LL}}(B)=1$) on a manifold $\mM$ and let ${\bf t}=(t_{\mu+1},\hdots,t_{\mu+m})\in\rr^m$ {(with a sufficiently small norm)}. We define the (possibly set-valued, in the case when $\emph{\textsf{\LL}}(B)=1$) \emph{multi}-\emph{flow mapping} $({\bf t},x)\mapsto\Psi_{B}^{{\bf g}}({\bf t})(x)$ associated with $B$ and ${\bf g}$, in a recursive manner as follows:
\begin{enumerate}[(i)]
\item{If $\emph{\textsf{\LL}}(B)=1$, i.e., $B=X_j$ for $j= \mu+1$, then we set
\begin{equation*}
\Psi_{B}^{{\bf g}}(t_j)(x)\colonequals e^{t_jg_j}(x)~\text{for every}~x\in\mM;\footnote{For a continuous vector field $g$, the notation $e^{tg}(\bar{x})$ stands for the set of values at $t\in\rr$ of all solutions of the Cauchy problem: $\dot{x}=g(x)$,~$x(0)=\bar{x}$, which is a singleton as soon as $g$ is locally Lipschitz continuous.}
\end{equation*}}
\item{If $\emph{\textsf{\LL}}(B)>1$ and $B=[B_1,B_2]$, with $\emph{\textsf{\LL}}(B_1)=m_1$ and $\emph{\textsf{\LL}}(B_2)=m_2$, then we set
\begin{equation*}
\Psi_B^{{\bf g}}({\bf t})(x)\colonequals\left(\Psi_{B_2}^{{\bf g}}({\bf t}_{2})\right)^{-1}\circ\left(\Psi_{B_1}^{{\bf g}}({\bf t}_{1})\right)^{-1}\circ\Psi_{B_2}^{{\bf g}}({\bf t}_{2})\circ\Psi_{B_1}^{{\bf g}}({\bf t}_{1})(x)~\text{for every}~x\in\mM,
\end{equation*}
where ${\bf t}_1=(t_{\mu+1},\hdots,t_{\mu+m_{1}})\in\rr^{m_1}$ and ${\bf t}_2=(t_{\mu+1+m_{1}},\hdots,t_{\mu+ m})\in\rr^{m_2}$.}
\end{enumerate}
\end{definition}

Before proceeding further, let us illustrate Definition \ref{commutator} with a few examples given below.
\begin{examples}
\begin{enumerate}[(i)]
\item{If $B\colonequals [X_1,X_2]$ is a formal bracket and ${\bf g}=(g_1,g_2)$ is a 2-tuple of vector fields of class $C^{B-1,1}$, then for every $x\in\mM$, the multi-flow map $\Psi_B^{{\bf g}}({\bf t})(x)$ is defined as follows:
\begin{equation*}
\Psi_B^{{\bf g}}(t_1,t_2)(x)\colonequals e^{-t_2g_2}\circ e^{-t_1g_1}\circ e^{t_2g_2}\circ e^{t_1g_1}(x).
\end{equation*}}
\item{Let $B\colonequals [X_2,[X_3,X_4]]$ be a formal bracket and let the 4-tuple ${\bf g}=(g_1,g_2,g_3,g_4)$ of vector fields be of class $C^{B-1,1}$, then for every $x\in\mM$, the multi-flow map $\Psi_B^{{\bf g}}({\bf t})(x)$ is defined as follows:
\begin{equation*}
\Psi_B^{{\bf g}}(t_2,t_3,t_4)(x)\colonequals \left(e^{-t_3g_3}\circ e^{-t_4g_4}\circ e^{t_3g_3}\circ e^{t_4g_4}\right)\circ e^{-t_2g_2}\circ\left(e^{-t_4g_4}\circ e^{-t_3g_3}\circ e^{t_4g_4}\circ e^{t_3g_3}\right)\circ e^{t_2g_2}(x).
\end{equation*}}
\item{If $B\colonequals [[X_2,X_3],[X_4,X_5]]$ is a formal bracket and ${\bf g}=(g_1,g_2,g_3,g_4,g_5)$ is a 5-tuple of vector fields of class $C^{B-1,1}$, then for every $x\in\mM$, the multi-flow map $\Psi_B^{{\bf g}}({\bf t})(x)$ is defined as follows:
\begin{align*}
\Psi_B^{{\bf g}}(t_2,t_3,t_4,t_5)(x)\colonequals &\left(e^{-t_4g_4}\circ e^{-t_5g_5}\circ e^{t_4g_4}\circ e^{t_5g_5}\right)\circ\left(e^{-t_2g_2}\circ e^{-t_3g_3}\circ e^{t_2g_2}\circ e^{t_3g_3}\right)\\
&\circ\left(e^{-t_5g_5}\circ e^{-t_4g_4}\circ e^{t_5g_5}\circ e^{t_4g_4}\right)\circ\left(e^{-t_3g_3}\circ e^{-t_2g_2}\circ e^{t_3g_3}\circ e^{t_2g_2}\right)(x).
\end{align*}}
\end{enumerate}
\end{examples}

The following result, that has been proven in \cite{Feleqi2017}, extends a classical result that is valid for smooth vector fields:
\begin{theorem}[An estimate for obtaining an asymptotic formula]\label{as-th}
Let $B$ be a formal bracket, with $\emph{\textsf{Seq}}(B)=X_{\mu+1}\hdots X_{\mu+m}$ for some integers $\mu\geq 0$ and $m\geq 1$. Furthermore, for an integer $r\geq\mu+m$, let $\mathbf{g}=(g_1,\hdots,g_{r})$ be an $r$-tuple of vector fields of class $C^{B-1,1}$ (respectively, $C^{B}$, if $\emph{\textsf{\LL}}(B)=1$) on $\rr^n$. Then, for any $x_*\in\rr^m$, there exists a real number $\delta>0$ and a modulus $\tilde{\omega}$, such that for any $({\bf t},x)\in\overline{\mathfrak{B}}_\delta(({\bf 0},x_*))$, the following inequality:
\begin{equation*}
\emph{dist}\big(y-x,t_{1}\cdots t_{m} B_{\emph{set}}({\bf g})(x_*)\big)\leq |t_{1}\cdots t_{m}|\tilde{\omega}(|(t_1,\hdots,t_m)|+|x-x_*|)
\end{equation*}
holds for every $y\in\Psi_B^{{\bf g}}({\bf t})(x)$, where ${\bf t}=(t_{1},\hdots,t_{m})\in\rr^m$.\footnote{As is customary, the distance from a point $x\in\rr^{n}$ to a set $K\subset\rr^{n}$, is defined as follows:
\begin{equation*}
\text{dist}\big(x,K\big)\colonequals\inf\left\{|x-y|\; :\; y\in K\right\}.
\end{equation*}}
\end{theorem}

The proof of Theorem~\ref{as-th}, in the case when $\text{\textsf{\LL}}(B)=1$, can be easily deduced from the proof of \cite[Theorem~4.1]{Rampazzo2001} (see also \cite[Lemma~5.5]{Bardi2020}), whereas, in the case when $\text{\textsf{\LL}}(B)>1$, this is \cite[Theorem~3.7]{Feleqi2017}. Before proceeding further, let us fix some more notations.
\begin{definition}
Let $x_*\in\mM$. For any given vector $w\in T_{x_*}\mM$, let us define the linear map $L^w\in\emph{\textsf{Lin}}(\rr\times T_{x_*}\mM,T_{x_*}\mM)$ as follows:
\begin{equation*}
L^w(s,v)\colonequals v+sw~\text{for every}~(s,v)\in\rr\times T_{x_*}\mM.
\end{equation*}
Moreover, for a non-empty set ${\bf W}\subset T_{x_*}\mM$, let us define the {set of linear maps} ${\bf L}^{\bf W}\subset\emph{\textsf{Lin}}(\rr\times T_{x_*}\mM,T_{x_*}\mM)$ as follows:
\begin{equation*}
{\bf L}^{\bf W}\colonequals\{L^w\; : \;w\in {\bf W}\}.
\end{equation*}
\end{definition}

The following result is of crucial use in proving the first claim stated in the main result (see Theorem~\ref{ChowRas}) of this paper:
\begin{theorem}\label{GDQ-lem}
For {given} integers {$r,\mu\ge 0$ and $m\geq 1$} such that {$r\geq\mu+m$}, let $B$ be a formal bracket such that {$\emph{\textsf{Seq}}(B)=X_{\mu+1}\hdots X_{\mu+m}$}, and let ${\bf g}=(g_1,\hdots,g_r)$ be an $r$-tuple of {vector} fields of class $C^{B-1,1}$ (respectively, $C^{B}$, if $\emph{\textsf{\LL}}(B)=1$) on a manifold $\mM$. Then, for every $x_*\in\mM$, the set ${\bf L}^{{B_{\emph{set}}}({\bf g})(x_*)}\subset\emph{\textsf{Lin}}(\rr\times T_{x_*}\mM,T_{x_*}\mM)$ is a GDQ of the mapping $(t,x)\mapsto\Sigma_B^{{\bf g}}(t)(x)$ at $((0,x_*),x_*)$ in the direction of $\rr\times\mM$, where for $x\in\mM$ and $\delta_x>0$ being a sufficiently small real number, the mapping
\begin{equation*}
[-\delta_x,\delta_x]\ni t\mapsto\Sigma_B^{{\bf g}}(t)(x)\in\mM,
\end{equation*}
is defined as follows:
\begin{enumerate}[(i)]
\item{If $\emph{\textsf{\LL}}(B)=1$, i.e., $B=X_j$ for $j=\mu+1$, then for all $t\in [-\delta_x,\delta_x]$, we set
\begin{equation*}
\Sigma_B^{{\bf g}}(t)(x)\colonequals e^{t f_j}(x);
\end{equation*}}
\item{If $\emph{\textsf{\LL}}(B)=m>1$ and $B=[B_1,B_2]$, with $\emph{\textsf{\LL}}(B_1)=m_1$ and $\emph{\textsf{\LL}}(B_2)=m_2$, then for all $t\in [-\delta_x,\delta_x]$, we set
\begin{align*}
&\Sigma_B^{{\bf g}}(t)(x)\colonequals\\
&\begin{cases}
\bigg(\Psi_{B_1}^{{\bf g}}\Big(\underbrace{|t|^{\frac{1}{m}},\hdots,|t|^{\frac{1}{m}}}_{\emph{$m_1$-times}}\Big)\bigg)^{-1}\circ\bigg(\Psi_{B_2}^{{\bf g}}\Big(\underbrace{|t|^{\frac{1}{m}},\hdots,|t|^{\frac{1}{m}}}_{\emph{$m_2$-times}}\Big)\bigg)^{-1}\circ\Psi_{B_1}^{{\bf g}}\Big(\underbrace{|t|^{\frac{1}{m}},\hdots,|t|^{\frac{1}{m}}}_{\emph{$m_1$-times}}\Big)\circ\Psi_{B_2}^{{\bf g}}\Big(\underbrace{|t|^{\frac{1}{m}},\hdots,|t|^{\frac{1}{m}}}_{\emph{$m_2$-times}}\Big)(x),\\\hfill \emph{if $m$ is even and $t<0$},\\
\Psi_B^{{\bf g}}\Big(\underbrace{-|t|^{\frac{1}{m}},\hdots,-|t|^{\frac{1}{m}}}_{\emph{$m$-times}}\Big)(x),\hfill \emph{if $m$ is odd and $t<0$},\hspace{3pt}\\
\Psi_B^{{\bf g}}\Big(\underbrace{t^{\frac{1}{m}},\hdots,t^{\frac{1}{m}}}_{\emph{$m$-times}}\Big)(x),\hfill \emph{if $t\geq 0$}.\hspace{64pt}
\end{cases}
\end{align*}}
\end{enumerate}
\end{theorem}

\begin{proof}
Since the conclusion is of local nature, without loss of generality, we can assume that $\mM=\rr^n$ for some integer $n>0$. When $\textsf{\LL}(B)=1$, the statement of Theorem~\ref{GDQ-lem} simplifies to the following:
\begin{enumerate}[$\bullet$]
\item\emph{Let $g$ be a continuous vector field on $\rr^n$. Then, for every $x_*\in\rr^m$, the set ${\bf L}^{\{g(x_*)\}}\subset\emph{\textsf{Lin}}(\rr\times \rr^m,\rr^m)$ is a GDQ of the (possibly set-valued) mapping $(t,x)\mapsto e^{tg}(x)$ at $((0,x_*),x_*)$ in the direction of $\rr\times\rr^m$.}
\end{enumerate}

The proof in this case now follows from \cite[Theorem~4.1]{Rampazzo2001}. Next, assume that $\textsf{\LL}(B)>1$. From Theorem~\ref{as-th}, it follows that there exists a modulus $\bar{\o}$ such that the inequality:
\begin{equation*}
\text{dist}\big(\Sigma_B^{{\bf g}}(t)(x)-x,t B_{\text{set}}({\bf g})(x_*)\big)\leq |t|\bar{\o}\Big(|t|^{\frac{1}{m}}+|x-x_*|\Big)
\end{equation*}
holds for every $(t,x)\in\mathcal{O}_{0}\times\mathcal{O}_{x_*}$, where $\mathcal{O}_{0}\times\mathcal{O}_{x_*}$ is a neighborhood of $(0,x_*)$ in $\rr\times\rr^n$. Now, fix any $(t,x)\in\mathcal{O}_{0}\times (\mathcal{O}_{x_*}\cap\mathfrak{B}_{1}(x_*))$\footnote{For a real number $r>0$ and an integer $n>0$, the notation $\mathfrak{B}_r(\bar{x})$ stands for the open ball in $\rr^n$ of radius $r$ centered at $\bar{x}\in\rr^n$.}. By using the facts that the mapping $\rr_{\geq 0}\ni z\mapsto z^{\frac{1}{m}}\in\rr_{\geq 0}$ \footnote{The set of all non-negative real numbers, is denoted by $\rr_{\geq 0}$.} is concave and the modulus $\bar{\o}$ is an increasing function, we now have the following chain of inequalities:
\begin{align*}
&\text{dist}\big(\Sigma_B^{{\bf g}}(t)(x)-x,t B_{\text{set}}({\bf g})(x_*)\big)\leq |t|\bar{\o}\Big(|t|^{\frac{1}{m}}+|x-x_*|^{\frac{1}{m}}\Big)\leq |t|\bar{\o}\Big(2^{1-\frac{1}{m}}(|t|+|x-x_*|)^{\frac{1}{m}}\Big)\\
&\leq |t|\bar{\o}\Big(2^{1-\frac{1}{m}}(|(t,x-x_*)|+|(t,x-x_*)|)^{\frac{1}{m}}\Big)=|t|\hat{\o}(|(t,x-x_*)|),
\end{align*}
where we have set $\hat{\o}(\cdot)\colonequals\bar{\o}\circ 2(\cdot)^{\frac{1}{m}}$. Hence, the following inequality:
\begin{equation*}
\inf\left\{\left|\Sigma_B^{{\bf g}}(t)(x)-\Sigma_B^{{\bf g}}(0)(x_*)-L(t,x-x_*)\right|\; :\; L\in{\bf L}^{B_{\text{set}}({\bf g})(x_*)}\right\}\leq |(t,x-x_*)|\hat{\o}(|(t,x-x_*)|)
\end{equation*}
holds for every $(t,x)\in\mathcal{O}_{0}\times (\mathcal{O}_{x_*}\cap\mathfrak{B}_{1}(x_*))$.
By invoking Corollary~\ref{GDQ-c}, we can now conclude that ${\bf L}^{{B_{\text{set}}}({\bf g})(x_*)}$ is a GDQ of $\Sigma_B^{{\bf g}}$ at $((0,x_*),x_*)$ in the direction of $\rr\times\rr^m$.
\end{proof}

\section{The Main Result}\label{s-Chow}
\subsection{A Deterministic Version of the $L^\infty$ Rashevskii-Chow Theorem}
We begin by establishing the main result of this paper under the following \emph{uniqueness hypothesis} {\bf (UH)}, which is verified, e.g., when the vector fields $\{f_1,\hdots,f_\nu\}$ are locally Lipschitz continuous.
\begin{equation*}
{\bf (UH)}~\left\lbrace
\begin{tabular}{@{}p{400pt}}
Let $K$ be a compact subset of $\mM$. Then, there exists $t>0$ such that, for each $j\in\{1,\hdots,\nu\}$ and for every $\bar{x}\in K$, the solution of the Cauchy problem: $\dot{x}=f_j(x),~x(0)=\bar{x}$ exists uniquely on the interval $[0,t]$.
\end{tabular}\right.
\end{equation*}

In order to state the main result of this paper --namely, Theorem~\ref{ChowRas} below-- we need the following definitions:
\begin{definition}[Reachable set]
Let $x_*\in \mM$ and let $t\ge0$. The \emph{reachable set from $x_*$ up to time $t$}, is defined as follows:
\begin{equation*}
\emph{\textsf{Reach}}(t,x_*)\colonequals\bigcup_{0\leq s\leq t}\{x_{x_*,u}(s)\; : \;{u}\colon [0,s]\to\rr^\nu~\emph{is an admissible control}\},
\end{equation*}
where $x_{x_*,u}\colon [0,t]\to\mM$ is the unique solution of \eqref{sistemaintro1} starting from $x(0)=x_*$ and corresponding to the control $u$.\footnote{Recall from the Introduction (see Section~\ref{s-intro}) that, for every time $t\ge 0$, an \emph{admissible control} $u=(u_1,\hdots,u_\nu)\colon [0,t]\to\rr^\nu$ is a piecewise constant map that takes values in the set of standard basis vectors $\{\pm {\bf e}_1,\hdots,\pm {\bf e}_\nu\}$ of $\rr^\nu$.}
\end{definition}

\begin{definition}[Small-time local controllability]
Let $x_*\in \mM$. Then, the control system \eqref{sistemaintro1} is said to be \emph{small}-\emph{time locally controllable} from $x_*$, if for {every} $t>0$, the reachable set $\emph{\textsf{Reach}}(t,x_*)$ is a neighborhood of $x_*$.
\end{definition}

\begin{definition}[Minimum-Time function]
Let $x_*\in \mM$. The \emph{minimum-time function} $T\colon\mM \to [0,+\infty]$ (from $x_*$), is defined, for every $x\in\mM$ as follows:
\begin{align*}
T(x)&\colonequals\inf\{t\geq 0\; : \;x_{x_*,u}(t)=x~\emph{for some admissible control}~u\colon [0,t]\to\rr^\nu\}\\
&=\inf\{t\geq 0\; : \;x\in\emph{\textsf{Reach}}(t,x_*)\}.\footnotemark
\end{align*}
\end{definition}
\footnotetext{One has $T(x)=+\infty$ for every $x\notin\bigcup_{t \ge 0}\text{\textsf{Reach}}(t,x_*)$.}

\begin{definition}[$L^\infty$ bracket-generating condition]\label{brackgen}
Let $\ell\geq 1$ be a given integer and let $B_1,\hdots,B_\ell$ be formal {iterated} brackets. For a given integer $\nu\geq 1$, we say that a finite family $\{g_1,\hdots,g_\nu\}$ of vector fields on a manifold $\mM$ is $L^\infty$ \emph{bracket-generating} at a point $x\in \mM$ with respect to the formal brackets $B_1,\hdots,B_\ell$, if there exists an integer $r\geq 1$, with $r\leq\nu$, such that the $r$-tuple ${\bf\hat g}=(g_1,\hdots,g_r)$ verifies the following conditions:
\begin{enumerate}[(i)]
\item{For every $j\in\{1,\hdots,{\ell}\}$, one has
\begin{equation*}
{\bf\hat g}\in\begin{cases}
C^{B_j}, &\emph{if \textsf{\LL}}(B_j)=1,\\
C^{B_j-1,1}, &\emph{if \textsf{\LL}}(B_j)>1;
\end{cases}
\end{equation*}}
\item{For {every} {$\ell$-tuple} $(v_1,\hdots,v_\ell)\in {(B_1)}_\emph{set}({\bf\hat g})({x})\times\cdots\times {(B_\ell)}_\emph{set}({\bf\hat g})({x})$, one has
\begin{equation*}
\emph{span}\left\{v_1,\hdots,v_\ell\right\}=T_x\mM.
\end{equation*}}
\end{enumerate}
\end{definition}

\begin{remark}
In item $(i)$ of Definition~\ref{brackgen}, if we assume instead that the {$r$-tuple $(g_1,\hdots,g_r)$ of vector fields} is {of class $C^{B_j}$} for every $j\in\{1,\hdots,\ell\}$, then we obtain a hypothesis which is still a bit weaker than the one used in \cite{Bramanti2013}. For instance, {if the formal bracket $[[X_1,X_2],[X_3,X_4]]$ were one of the formal iterated brackets $B_1,\hdots,B_\ell$}, then our hypothesis would require that the vector fields $\{f_1,f_2,f_3,f_4\}$ be of class $C^2$, while the hypothesis used in \cite{Bramanti2013} (and called ``non-smooth'' therein) would require these to be at least of class $C^3$.
\end{remark}

We are now ready to state the main result:
\begin{theorem}[A deterministic $L^\infty$ Rashevskii-Chow theorem]\label{ChowRas} 
Let $\ell\geq 1$ be a given integer and let $B_1,\hdots,B_\ell$ be formal iterated brackets. For a given integer $\nu\geq 1$, if the family $\{f_1,\hdots,f_\nu\}$ of continuous vector fields on a manifold $\mM$ verifies the uniqueness hypothesis {\bf (UH)} and is also $L^\infty$ bracket-generating at $x_*\in \mM$ with respect to the formal brackets $B_1,\hdots,B_\ell$, {then} the control system \eqref{sistemaintro1} is STLC from $x_*$. Moreover, the minimum-time function $T\colon\mathcal{O}\to [0,+\infty]$ verifies $T(x)\leq C\,d(x,x_*)^{1/r}$ for all $x$ in some neighborhood $\mathcal{O}$ of $x_*$ (endowed with a Riemannian distance $d\colon\mathcal{O}\times\mathcal{O}\to\rr$), with $C>0$ being a suitable constant and ${r}\colonequals\max\left\{\emph{\textsf{\LL}}(B_j)\; : \;j\in\{1,\hdots,\ell\}\right\}$.
\end{theorem}

\begin{remark}\label{literature}
{As we have already pointed out in Section~\ref{s-intro}, the interest in STLC goes far beyond than just the realm of control theory. See, e.g., \cite{Montgomery2002,Agrachev2019} and the references therein, for the use of the bracket-generating condition in \emph{Sub}-\emph{Riemannian geometry}, where the trajectories of $\eqref{sistemaintro1}$ are called \emph{horizontal curves}. For the use of the bracket-generating condition in the theory of \emph{degenerate elliptic} and \emph{parabolic partial differential equations}, see, e.g., \cite{Hormander1967,Bony1969,Bardi1999,Bramanti2013,Bramanti2014} and the references therein. For the use of the bracket-generating condition in the study of \emph{eikonal equations}, see, e.g., \cite{Bardi2020} and the references therein.}
\end{remark}

Before proving Theorem~\ref{ChowRas}, let us illustrate its application to the following example:
\begin{example}\label{esempio}
Consider the following control system on $\rr^4$:
\begin{equation}\label{sistemaesempio}
\dot x=\sum_{i=1}^3 u_if_i(x),
\end{equation}
with control constraints $|u_i|\leq 1$, $i=1,2,3$ and the vector fields (in the standard coordinates on $\rr^4$):
\begin{equation*}
f_1(x)\colonequals\frac{\partial}{\partial x_2}+\frac{\partial}{\partial x_4},~f_2(x)\colonequals\frac{\partial}{\partial x_1} +(2x_2^2+ x_2|x_2|)\frac{\partial}{\partial x_3},~f_3(x)\colonequals\phi(x_2)\frac{\partial}{\partial x_4},\footnote{As is customary, for any choice of local coordinates $(x_1,\hdots,x_n)$ around a point $P$ of a $n$-dimensional manifold $\mM$, $\left\{\frac{\partial}{\partial x_1},\hdots,\frac{\partial}{\partial x_n}\right\}$ denotes the standard basis of the tangent space $T_{P}\mM$.}
\end{equation*}
where {$\phi: \rr\to\rr$ is a nowhere vanishing continuous function.}
A simple computation yields:
\begin{equation*}
[f_1,f_2](x)=(4x_2+2|x_2|)\frac{\partial}{\partial x_3}.
\end{equation*}
Hence, at all the points of the hyperplane $x_2=0$, the family $\{f_1,f_2,f_3\}$ of vector fields fails to be bracket-generating of step $2$. However, one can verify that
\begin{equation*}
[f_1,[f_1,f_2]]_{\text{set}}(x)=\left\{\alpha\frac{\partial}{\partial x_3}\; : \;\alpha\in [2,6]\right\},
\end{equation*}
which results in
\begin{equation*}
\text{span}\left\{f_1(x),f_2(x),f_3(x),[f_1,f_2](x),v\right\}=\rr^4
\end{equation*}
for all $v\in [f_1,[f_1,f_2]]_{\text{set}}(x)$ and $x\in\rr^4$. Therefore, the family $\{f_1,f_2,f_3\}$ of vector fields is $L^\infty$ bracket-generating at every point of $\rr^4$ with respect to the formal brackets $B_1\colonequals X_1,B_2\colonequals X_2,B_3\colonequals X_3,B_4\colonequals [X_1,X_2],B_5\colonequals [X_1,[X_1,X_2]]$. Using Theorem~\ref{ChowRas}, we can now conclude that the control system \eqref{sistemaesempio} is STLC from every point of $\rr^4$. In addition, for any choice of a point $x_*\in\rr^{4}$, the minimum-time function $T: \mathcal{O}_{x_*}\to [0,+\infty]$, verifies $T(x)\leq C_{x_*}|x-x_*|^{1/3}$ for all $x$ in some neighborhood $\mathcal{O}_{x_*}$ of $x_*$, with $C_{x_*}>0$ being a suitable constant.
\end{example}

\begin{proof}[Proof of Theorem~\ref{ChowRas}]
Since the conclusion is of local nature, without loss of generality, we can assume that $\mM=\rr^n$ for some integer $n>0$. Furthermore, we can also assume that $d(x_1,x_2)\colonequals |x_1-x_2|$ for every $x_1,x_2\in\rr^n$. To prove the claim regarding the small-time local controllability of the control system \eqref{sistemaintro1} from $x_*$, let us begin by observing that ${\bf\hat g}$ in Definition~\ref{brackgen} here coincides with the $r$-tuple of vector fields $(f_1,\hdots,f_r)$. Fixing the same notations as given in Definition~\ref{commutator}, let us now recall from Theorem~\ref{GDQ-lem} the definition of the mapping
\begin{equation*}
[-\delta_x,\delta_x]\ni t\mapsto\Sigma_B^{{\bf\hat g}}(t)(x)\in\rr^n,
\end{equation*}
which for $x\in\rr^n$ and $\delta_x>0$ being a sufficiently small real number, is given as follows:
\begin{enumerate}[(i)]
\item{If $\text{\textsf{\LL}}(B)=1$, i.e., $B=X_j$ for some $j=\mu+1$, then for all $t\in [-\delta_x,\delta_x]$, we set
\begin{equation*}
\Sigma_B^{{\bf\hat g}}(t)(x)\colonequals e^{t f_j}(x);
\end{equation*}}
\item{If $\text{\textsf{\LL}}(B)=m>1$ and $B=[B_1,B_2]$, with $\text{\textsf{\LL}}(B_1)=m_1$ and $\text{\textsf{\LL}}(B_2)=m_2$, then for all $t\in [-\delta_x,\delta_x]$, we set
\begin{align*}
&\Sigma_B^{{\bf\hat g}}(t)(x)\colonequals\\
&\begin{cases}
\bigg(\Psi_{B_1}^{{\bf\hat g}}\Big(\underbrace{|t|^{\frac{1}{m}},\hdots,|t|^{\frac{1}{m}}}_{\text{$m_1$-times}}\Big)\bigg)^{-1}\circ\bigg(\Psi_{B_2}^{{\bf\hat g}}\Big(\underbrace{|t|^{\frac{1}{m}},\hdots,|t|^{\frac{1}{m}}}_{\text{$m_2$-times}}\Big)\bigg)^{-1}\circ\Psi_{B_1}^{{\bf\hat g}}\Big(\underbrace{|t|^{\frac{1}{m}},\hdots,|t|^{\frac{1}{m}}}_{\text{$m_1$-times}}\Big)\circ\Psi_{B_2}^{{\bf\hat g}}\Big(\underbrace{|t|^{\frac{1}{m}},\hdots,|t|^{\frac{1}{m}}}_{\text{$m_2$-times}}\Big)(x),\\\hfill \text{if $m$ is even and $t<0$},\\
\Psi_B^{{\bf\hat g}}\Big(\underbrace{-|t|^{\frac{1}{m}},\hdots,-|t|^{\frac{1}{m}}}_{\text{$m$-times}}\Big)(x),\hfill \text{if $m$ is odd and $t<0$},\hspace{3pt}\\
\Psi_B^{{\bf\hat g}}\Big(\underbrace{t^{\frac{1}{m}},\hdots,t^{\frac{1}{m}}}_{\text{$m$-times}}\Big)(x),\hfill \text{if $t\geq 0$}.\hspace{64pt}
\end{cases}
\end{align*}}
\end{enumerate}

For a sufficiently small real number $\delta>0$, let us now define the map ${\bf x}^{(B_1({\bf\hat g}),\hdots,B_\ell({\bf\hat g}))}\colon\mathfrak{B}_{\delta}({\bf 0})\to \rr^n$, by setting
\begin{equation*}
{\bf x}^{(B_1({\bf\hat g}),\hdots,B_\ell({\bf\hat g}))}(t_1,\hdots,t_\ell)\colonequals\Sigma_{B_\ell}^{{\bf\hat g}}(t_\ell)\circ\cdots\circ\Sigma_{B_1}^{{\bf\hat g}}(t_1)(x_*)
\end{equation*}
for every $(t_1,\hdots,t_\ell)\in\mathfrak{B}_{\delta}({\bf 0})$, where ${\bf 0}$ denotes the zero vector in $\rr^\ell$. Let us also now define the subset $\mathbf{A}^{(B_1({\bf\hat g}),\hdots,B_\ell({\bf\hat g}))}\subset\textsf{Lin}(\rr^\ell,\rr^m)$ as follows:
\begin{equation*}
\mathbf{A}^{(B_1({\bf\hat g}),\hdots,B_\ell({\bf\hat g}))}\colonequals\{A^{(v_1,\hdots,v_\ell)}\; : \;v_j\in (B_j)_\text{set}({\bf\hat g})(x_*),~j\in\{1,\hdots,\ell\}\},
\end{equation*}
where for any $(v_1,\hdots,v_\ell)\in\rr^m\times\cdots\times\rr^m$, the linear map $A^{(v_1,\hdots,v_\ell)}\in\textsf{Lin}(\rr^\ell,\rr^m)$, is defined as follows:
\begin{equation*}
A^{(v_1,\hdots,v_\ell)}(t_1,\hdots,t_\ell)\colonequals\sum_{i=1}^{\ell}L^{v_i}(t_i,{\bf 0})=\sum_{i=1}^{\ell}t_iv_i
\end{equation*}
for every $(t_1,\hdots,t_\ell)\in\rr^\ell$. By using the chain rule (see Theorem~\ref{chain}), together with Theorem~\ref{GDQ-lem}, we can now deduce that ${\bf A}^{(B_1({\bf\hat g}),\hdots,B_\ell({\bf\hat g}))}$ is a {GDQ} of ${\bf x}^{(B_1({\bf\hat g}),\hdots,B_\ell({\bf\hat g}))}$ at $({\bf 0},{\bf x}^{(B_1({\bf\hat g}),\hdots,B_\ell({\bf\hat g}))}({\bf 0}))=({\bf 0},{x_*})$ in the direction of $\rr^\ell$. By hypothesis, any linear map $A^{(v_1,\hdots,v_\ell)}\in\mathbf{A}^{(B_1({\bf\hat g}),\hdots,B_\ell({\bf\hat g}))}$ is surjective, i.e., $A^{(v_1,\hdots,v_\ell)}(\rr^\ell)=\rr^m$. In particular, this means that the map ${\bf x}^{(B_1({\bf\hat g}),\hdots,B_\ell({\bf\hat g}))}$ satisfies all the assumptions of the open mapping theorem (see Theorem~\ref{open-t}), using which we can conclude that the control system \eqref{sistemaintro1} is STLC from $x_*$.

To prove the claim regarding the H{\"o}lder estimate of the minimum-time function at $x_*$, let us begin by observing that, from the ``linear rate'' property stated in Theorem~\ref{open-t}, there exist constants $C_1>0$ and $\delta_1>0$ such that
\begin{equation}\label{reach-1}
\mathfrak{B}_{C_1\theta}(x_*)\subset{\bf x}^{(B_1({\bf\hat g}),\hdots,B_\ell({\bf\hat g}))}(\mathfrak{B}_\theta({\bf 0}))
\end{equation}
for all $\theta\in (0,\delta_1]$. For any formal bracket $B$, let us now define the natural number $n(B)$, recursively as follows:
\begin{equation*}
n(B)\colonequals\begin{cases}
1, &\text{if $\text{\textsf{\LL}}(B)=1$},\\
2n({\invbreve B})+2n({\breve B}), &\text{if $\text{\textsf{\LL}}(B)>1$}~\text{and}~\text{$B=[{\invbreve B},{\breve B}]$}.
\end{cases}
\end{equation*}
Note that, for a sufficiently small real number $\delta>0$, if we define the function $\tau^{(B_1({\bf\hat g}),\hdots,B_\ell({\bf\hat g}))}\colon\mathfrak{B}_{\delta}({\bf 0})\to\rr$, by setting
\begin{equation*}
\tau^{(B_1({\bf\hat g}),\hdots,B_\ell({\bf\hat g}))}(t_1,\hdots,t_\ell) \colonequals\sum_{i=1}^\ell n({B_i})|t_i|^{\frac{1}{\text{\textsf{\LL}}(B_i)}}
\end{equation*}
for all $(t_1,\hdots,t_\ell)\in\mathfrak{B}_\delta({\bf 0})$, then ${\bf x}^{(B_1({\bf\hat g}),\hdots,B_\ell({\bf\hat g}))}(t_1,\hdots,t_\ell)$ is nothing but the value at time $\tau^{(B_1({\bf\hat g}),\hdots,B_\ell({\bf\hat g}))}(t_1,$ $\hdots,t_\ell)$ of a solution to \eqref{sistemaintro1}, starting from $x_*$ at time $0$ and corresponding to a suitable admissible control $(u_1,\hdots,u_\nu)\colon [0,\bar\tau]\to\rr^\nu$, with $\bar\tau=\tau^{(B_1({\bf\hat g}),\hdots,B_\ell({\bf\hat g}))}(t_1,\hdots,t_\ell)$. Keeping this fact in mind and noting that the inequality:
\begin{equation*}
\tau^{(B_1({\bf\hat g}),\hdots,B_\ell({\bf\hat g}))}(t_1,\hdots,t_\ell)\leq\bar C\left(\max\left\{|t_j|^{\frac{1}{r}}\; : \;j\in\{1,\hdots,\ell\}\right\}\right)
\end{equation*}
holds for every $(t_1,\hdots,t_\ell)\in [0,1]\times\cdots\times [0,1]$, where the constant $\bar C\colonequals\big(\sum_{i=1}^\ell n({B_i})\big)\geq 1$, we can now deduce that there exist some constants $C_2>0$ and $\delta_2\in (0,\delta_1]$ such that
\begin{equation}\label{reach-2}
{\bf x}^{(B_1({\bf\hat g}),\hdots,B_\ell({\bf\hat g}))}(\mathfrak{B}_\theta({\bf 0}))\subset\textsf{Reach}\big(C_2\theta^{\frac{1}{r}},x_*\big)
\end{equation}
for every $\theta\in (0,\delta_2]$.
From \eqref{reach-1} and \eqref{reach-2}, we have that
\begin{equation*}
\mathfrak{B}_{C_3\theta}(x_*)\subset\textsf{Reach}\big(\theta^{\frac{1}{r}},x_*\big)
\end{equation*}
for every $\theta\in (0,\delta_2]$, where $C_3>0$ is a suitable constant. {Hence, for any $\theta\in (0,\delta_2]$ and $x\in\mathfrak{B}_{C_3\theta}(x_*)$, we have that $x\in\textsf{Reach}\big(\theta^{\frac{1}{r}},x_*\big)$, and as a result, the following inequality:
\begin{equation*}
T(x)\leq\theta^{\frac{1}{r}}\leq C_3^{-\frac{1}{r}}|x-x_*|^{1/r}
\end{equation*}
holds for every $x\in\mathfrak{B}_{C_3\theta}(x_*)$, which completes the proof.}
\end{proof}

\subsection{A Non-Deterministic Version of the $L^\infty$ Rashevskii-Chow Theorem}
We now establish a result that generalizes Theorem~\ref{ChowRas}, in that the uniqueness hypothesis {\bf (UH)} is not assumed. In such a case, the notions of a ``reachable set'', ``small-time local controllability'' and a ``minimum-time function'', clearly fail to have any deterministic meaning. However, a ``non-deterministic version'' of Theorem~\ref{ChowRas} (see Theorem~\ref{wChowRas} below) can still be established. In order to do so, let us introduce a weak version of a reachable set, small-time local controllability and a minimum-time function.

\begin{definition}[Weak reachable set]
Let $x_*\in \mM$ and let $t\ge0$. The set
\begin{equation*}
\emph{\textsf{wReach}}(t,x_*)\colonequals\bigcup_{0\leq s\leq t}\{{\mathtt x}_{x_*,u}(s)\; : \;{u}\colon [0,s]\to\rr^\nu~\emph{is an admissible control}\},
\end{equation*}
where the (possibly set-valued) map ${\mathtt x}_{x_*,u}\colon [0,t]\to\mM$, is defined by
\begin{equation*}
{\mathtt x}_{x_*,u}(s)\colonequals\Big\{x_{x_*,u}(s)\; : \;x_{x_*,u}(\cdot)~\emph{is a solution of}~\eqref{sistemaintro1}~\emph{s.t.}~x(0)=x_*,~\emph{corresponding to an admissible control}~u\Big\},\footnotemark
\end{equation*}
is called the \emph{weak reachable set from $x_*$ up to time $t$.}
\end{definition}
\footnotetext{Notice that, since the vector fields $\{f_1,\hdots,f_\nu\}$ in \eqref{sistemaintro1} are continuous, there exists a real number $\bar{t}\in (0,t]$ such that the set ${\mathtt x}_{x_*,u}(s)\neq\{\emptyset\}$ for every $s\in [0,\bar{t}]$ and every admissible control $u$.}

\begin{definition}[Weak small-time local controllability]
Let $x_*\in \mM$. The control system \eqref{sistemaintro1} is said to be \emph{weakly small}-\emph{time locally controllable} from $x_*$, if for {every} $t>0$, the weak reachable set $\emph{\textsf{wReach}}(t,x_*)$ is a neighborhood of $x_*$.
\end{definition}

\begin{definition}[Weak minimum-time function]
Let $x_*\in \mM$. The \emph{weak minimum-time function} $T_w\colon\mM\to [0,+\infty]$ (from $x_*$), is defined, for every $x\in\mM$ as follows:
\begin{align*}
T_w(x)&\colonequals\inf\{t\geq 0\; : \;{\mathtt x}_{x_*,u}(t)\ni x~\emph{for some admissible control}~{u}\colon [0,t]\to\rr^\nu\}\\
&=\inf\{t\geq 0\; : \;x\in\emph{\textsf{wReach}}(t,x_*)\}.\footnotemark
\end{align*}
\end{definition}
\footnotetext{One has $T_w(x)=+\infty$ for every $x\notin\bigcup_{t\ge0}\text{\textsf{wReach}}(t,x_*)$.}

We are now ready to state the following result:
\begin{theorem}[A non-deterministic version of the $L^\infty$ Rashevskii-Chow theorem]\label{wChowRas}
Let $\ell\geq 1$ be a given integer and let $B_1,\hdots,B_\ell$ be formal iterated brackets. For a given integer $\nu\geq 1$, if the family $\{f_1,\hdots,f_\nu\}$ of continuous vector fields on a manifold $\mM$ is $L^\infty$ bracket-generating at $x_*\in \mM$ with respect to the formal brackets $B_1,\hdots,B_\ell$, {then} the control system \eqref{sistemaintro1} is weakly STLC from $x_*$. Moreover, the weak minimum-time function $T_w\colon\mathcal{O}\to [0,+\infty]$, verifies $T_w(x)\leq C\,d(x,x_*)^{1/r}$ for all $x$ in some neighborhood $\mathcal{O}$ of $x_*$ (endowed with a Riemannian distance $d\colon \mathcal{O}\times\mathcal{O}\to\rr$), with $C>0$ being a suitable constant and ${r}\colonequals\max\left\{\emph{\textsf{\LL}}(B_j)\; : \;j\in\{1,\hdots,\ell\}\right\}$.
\end{theorem}

\begin{proof}
The proof follows exactly the same steps as in the proof of Theorem~\ref{ChowRas}, but now one also has to account for an additional case (which is not covered by Theorem~\ref{ChowRas}), when the set $\mathcal{J}\colonequals\{j\in\{1,\hdots,\ell\}\; : \; \text{\textsf{\LL}}(B_j)=1\}$ is non-empty and also the flows of the vector fields associated with the formal brackets $B_j$, $j\in\mathcal{J}$, may happen to be multi-valued. However, Theorem~\ref{GDQ-lem} still applies and so the proof is concluded.
\end{proof}

\appendix
\appsection{Proof of Theorem~\ref{open-t}}\label{a-OMT}
For the proof of Thoerem~\ref{open-t}, we need an extension of Kakutani's fixed-point theorem due to Cellina. For the sake of completeness, we also provide the proof of this result.
\begin{theorem}[Cellina \cite{Cellina1970}]\label{Kak-ext}
Let $K$ be a non-empty compact and convex set of $\rr^n$ for some integer $m>0$, and let $\Phi \colon K \rightrightarrows K$ be a CCA set-valued map. Then, $\Phi$ has a fixed point, i.e., there exists $x\in K$ such that $x \in \Phi(x)$.
\end{theorem}

\begin{proof}
According to Definition~\ref{reg-def}, let $\{\varphi_j\}_{j\in \nn}$ be a sequence of single-valued continuous maps from $K$ to $K$ that inward graph converges to $\Phi$ as $j\to \infty$. By Brouwer's fixed-point theorem, there exists $x_j \in K$ such that $\varphi_j(x_j) = x_j$ for every $j\in \nn$. Since $K$ is compact, by passing to a subsequence if necessary, we can assume that there exists $x\in K$ such that $x_j \to x$ as $j\to \infty$. By the definition of inward graph convergence, we know that for every real number $\varepsilon>0$, $(x_j, \varphi(x_j)) = (x_j, x_j)$ belongs to the $\ve$-neighborhood of $\textsf{Gr}(\Phi)$:
\begin{equation*}
\textsf{Gr}_\ve(\Phi) \colonequals \Big\{(x,y) \in \rr^n \times \rr^n \: : \; 
\inf_{(v,w) \in \textsf{Gr}(\Phi)} (|x - v| + |y - w|) < \ve\Big\}
\end{equation*}
for each sufficiently large $j\in\nn$. In particular, for every $k\in\nn$, $(x_j, x_j) \in \textsf{Gr}_{1/k}(\Phi)$ for $j$ sufficiently large. Therefore, for every $k\in\nn$ there exist $j_k \in \nn$ and $v_k, w_k \in K$ such that $w_k \in \Phi(v_k)$ and
\begin{equation}\label{aux-ineq}
|x_{j_k} - v_k| + |x_{j_k} - w_k| < \frac{1}{k}.
\end{equation}
Since $\textsf{Gr}(\Phi)$ is compact, up to a subsequence, $(v_k, w_k) \to (v, w)$ as $k\to \infty$ for some $(v,w) \in \textsf{Gr}(\Phi)$. Hence, by passing to the limit as $k \to \infty$ in \eqref{aux-ineq}, we conclude that $x = v = w$ and thus $x \in \Phi(x)$.
\end{proof}

\begin{proof}[Proof of Theorem~\ref{open-t}]
Up to translations in $x$ and $y$, we can assume without loss of generality that $(x_*,y_*)=({\bf 0},{\bf 0})$. Since the set of surjective linear maps is open in $\text{\textsf{Lin}}(\rr^n, \rr^m)$\footnote{This fact is not too hard to verify for the finite-dimensional case (see, e.g., \cite[Theorem~3.4, Chapter~XV]{Lang1993} for the infinite-dimensional version of this fact).} and contains the compact set $\L$, let the real number $\delta>0$ be so small, such that the $\delta$-neighborhood $\L^\d$ of $\L$, also consists of surjective linear maps. With respect to this $\delta$, by the definition of a GDQ (see Definition~\ref{GDQ-def}) there exist $U$, a compact neighborhood of ${\bf 0}$ in $\rr^n$, and $G\colon U \rightrightarrows \L^\delta$, a CCA set-valued map, such that $G(x)x \subset F(x)$ for every $x\in U$.

It is easy to verify that a linear map $L\in\text{\textsf{Lin}}(\rr^n, \rr^m)$ is surjective if and only if $LL^{\sf T}$ is invertible, where $L^{\sf T}$ denotes the transpose (or adjoint) of $L$ and in such a case, the linear map $L^{\#} \colonequals L^{\sf T} (LL^{\sf T})^{-1}$ is a right inverse of $L$, i.e., $LL^{\#} = I$ (note that $L^{\#}$ is the Moore-Penrose pseudoinverse of $L$). Since every $L\in G(x)$ is surjective for all $x\in U$, the set-valued map $G^{\#} \colon U \rightrightarrows \text{\textsf{Lin}}(\rr^m, \rr^n)$ defined by
\begin{equation*}
G^{\#}(x) \colonequals \{L^{\#}\; : \; L \in G(x)\}
\end{equation*}
is well-defined and CCA, since it is the composite of the continuous (in fact, analytic) single-valued mapping $\L^\d \ni L \mapsto L^{\#} \in \text{\textsf{Lin}}(\rr^m, \rr^n)$ with the CCA set-valued map $G$.

Let the real number $\kappa>0$ be such that $| L^{\#} | \le \kappa$ for every $L \in \L^\delta$ (such a $\kappa$ exists, since $\L^\d$ is compact and the mapping $L \mapsto L^{\#}$ is continuous on $\L^\d$). Moreover, let the real number $\bve>0$ be such that $\overline{\mathfrak{B}}_{\kappa\bve}({\bf 0}) \subset U$. Now, for any real number $\ve \in (0,\bve]$ and $y \in \rr^m$ with $|y| \le \ve$, we want to show that there exists $x\in \rr^m$ with $|x| \le \kappa \ve$ such that $y \in F(x)$. To this end, let us consider the set-valued mapping:
\begin{equation*}
x \mapsto G^{\#}(x)y \colonequals \{L^{\#}y\; : \; L \in G(x)\}, 
\end{equation*}
defined on $U$, which can be verified to be CCA. Since, for $x\in\rr^n$ with $|x| \le \kappa\,\ve$ and $L\in G(x)$, one has that
\begin{equation*}
|L^{\#}y| \le k|y| \le k\ve,
\end{equation*}
the set-valued mapping $x\mapsto G^{\#}(x)y$, transforms the closed ball $\overline{\mathfrak{B}}_{k\ve}({\bf 0})$ into itself, i.e., $G^{\#}(x)y \subset \overline{\mathfrak{B}}_{k\ve}({\bf 0})$ for every $x\in \overline{\mathfrak{B}}_{k\ve}({\bf 0})$. By the extension of Kakutani's fixed-point theorem to CCA maps (see Theorem~\ref{Kak-ext}), there exists $x\in \overline{\mathfrak{B}}_{k\ve}({\bf 0})$ such that $x\in G^{\#}(x)y$, i.e., $x = L^{\#}y$ for some $L \in G(x)$. Therefore, since $Lx = LL^{\#}y = y$, we have that $y \in G(x)x \subset F(x)$.
\end{proof}

\appsection{Proof of Theorem~\ref{sucole}}\label{a-GDQ}
Since the conclusion is of local nature, without loss of generality, we can assume that $\mM=\rr^n$ for some integer $n>0$ and $(x_*,y_*)=({\bf 0},{\bf 0})$. Furthermore, as a distance we can take $d(x_1,x_2)\colonequals |x_1-x_2|$ for every $x_1,x_2\in V\colonequals \rr^n$. Let us now fix a real number $\d>0$, and let us choose $\bar\rho$ such that $0<\bar\rho\leq r$ and $\tilde{\o}([0,\bar\rho])\subset\left[0,\d/2\right]$. Let us also define the neighborhood $U\colonequals\overline{\mathfrak{B}}_{\bar\rho}({\bf 0})\subset V$. For each $j\in\nn$, let us now define the set-valued map {$G_j\colon U\rightrightarrows\Lambda^\delta$} as follows:
\begin{equation*}
G_j(x)\colonequals\{L\in\L^{\d}\; : \;\text{there exists}~y\in F_j(x)~\text{such that}~y=Lx\}.
\end{equation*}
Let us now show that for each sufficiently large $j\in\nn$, $G_j(x)\neq\{\emptyset\}$ for every $x\in U$. To this end, fix any $x\in U$, and for each $j\in\nn$, let us choose a point $y_j\in F_j(x)$ and a linear map $L_{1j}\in\L$ verifying $|y_j-L_{1j}x|\leq|x|\tilde{\o}_j(|x|)$. Such $y_j$ and $L_{1j}$ do exist in view of hypothesis $(iii)$, and also the compactness of $F_j(x)$ and $\Lambda$. Moreover, from the assumption of uniform convergence of the sequence $\{\tilde{\o}_j\}_{j\in\nn}$ to $\tilde{\o}$, there exists $\bar j\in\nn$ such that for each $j\geq\bar j$, one has that $\tilde{\o}_j(\rho)\leq|\tilde{\o}_j(\rho)-\tilde{\o}(\rho)|+|\tilde{\o}(\rho)|\leq\d/2+\d/2=\d$ for any $\rho\in [0,\bar\rho]$, which, in turn, implies that
\begin{equation*}
|y_j-L_{1j}x|\leq|x|\d
\end{equation*}
for each $j\geq\bar j$. Let us also define the linear map $L_{2j}\in\textsf{Lin}(\rr^n,\rr^m)$, by setting for every $z\in\rr^n$,
\begin{equation*}
L_{2j}z\colonequals\begin{cases}
\dfrac{\langle x,z\rangle}{|x|^2}(y_j-L_{1j}x), &\text{if $x\in\rr^n\setminus\{{\bf 0}\}$},\\
~~{\bf 0}, &\text{if $x={\bf 0}$}.
\end{cases}
\end{equation*}
If we now define the linear map ${L}_j\colonequals (L_{1j}+L_{2j})\in\textsf{Lin}(\rr^n,\rr^m)$, then we have that ${L}_j\in G_j(x)$ for each $j\geq\bar j$. Indeed, for each $j\geq\bar j$, one has that $L_{1j}\in\L$ and $|L_{2j}|\leq |y_j-L_{1j}x|/|x|\leq\d$ (which implies that ${L}_j\in\Lambda^\d$), $y_j\in F_j(x)$, and $y_j=L_jx$. Hence, for each $j\geq\bar j$, $G_j(x)\neq\{\emptyset\}$ for every $x\in U$.

Let us now define for every $x\in U$, the set-valued map {$G\colon U\rightrightarrows\Lambda^\delta$} as follows:
\begin{equation*}
G(x)\colonequals\{L\in\L^{\d}\; : \;\text{there exists}~y\in F(x)~\text{such that}~y=Lx\}.
\end{equation*}
Let us first show that $G(x)\neq\{\emptyset\}$ for every $x\in U$. Indeed, for any given $x\in U$, the sequence $\{L_j\}_{j\in\nn}$ constructed above is contained in $\Lambda^\d$, which is compact, so that, by possibly passing to a subsequence (which we do not relabel), $\{L_j\}_{j\in\nn}$ converges to a linear map $\tilde L\in \Lambda^\d$. Hence, $\lim_{j\to\infty}y_j=\lim_{j\to\infty}L_jx=\tilde Lx\equalscolon\tilde y$, and since $y_j\in F_j(x)$ and the sequence $\{F_j\}_{j\in\nn}$ is assumed to inward graph converge to $F$, we have that $\tilde y\in F(x)$. Hence, $\tilde L\in G(x)$, so that $G(x)\neq\{\emptyset\}$ for every $x\in U$.

By definition, $G$ satisfies condition $(iii)$ in Definition \ref{GDQ-def}. In order to verify that $G$ is {CCA}, let us begin by proving that for any non-empty compact set $K\subset U$, $\textsf{Gr}(G|_K)$ is compact. Since it is bounded, it suffices to show that it is closed. To this end, if $\{(x_i,L_i)\}_{i\in\nn}\subset\textsf{Gr}(G|_K)$ is a sequence converging to $(\tilde{x},\tilde{L})\in K\times\L^\d$, then by setting $y_i\colonequals L_ix_i$ for each $i\in\nn$, one has that $\lim_{i\to\infty}y_i=\tilde{y}\colonequals\tilde{L}\tilde{x}\in F(\tilde{x})$, since by assumption $F$ is upper semi-continuous and takes closed values. Hence, $(\tilde x,\tilde L)\in\textsf{Gr}(G|_K)$ and it now follows that $\textsf{Gr}(G|_K)$ is compact. Furthermore, arguing as above, one can verify that for any non-empty compact set $K\subset U$, $\textsf{Gr}(G_j|_K)$ is compact for each $j\geq\bar j$, so that, $G_j|_K$ is upper semi-continuous for each $j\geq\bar j$ (see \cite[Corollary~1.1.1, p. 42]{Aubin1984}). To verify that $G$ is a limit (in the sense of inward graph convergence) of a sequence of continuous single-valued maps from $K$ to $\Lambda^\delta$, where $K\subset U$ is any non-empty compact set, we first note that the sequence $\{G_j|_K\}_{j\in\nn}$ inward graph converges to $G|_K$. In particular, a suitable subsequence $\{G_{j_k}|_K\}_{k\in\nn}$ of $\{G_j|_K\}_{j\in\nn}$ verifies
\begin{equation*}
\sup\left\{\text{dist}(z,\textsf{Gr}(G|_K))\; : \;z\in\textsf{Gr}(G_{j_k}|_K)\right\}\leq\frac{1}{2k}
\end{equation*}
for each $k\in\nn$. Moreover, it is not too hard to verify that $G_{j}|_K$ is convex-valued for each $j\geq\bar j$. Hence, for each $j\geq\bar j$, {Cellina's} approximate selection theorem (see \cite[Theorem~1.12.1, p. 84]{Aubin1984}) implies that there exists a sequence $\{g_{ij}\}_{i\in\nn}$ of single-valued locally Lipschitz continuous functions from $K$ to $\L^\d$, which inward graph converges to $G_j|_K$. In particular, a suitable subsequence $\{g_{{i_k}{j_k}}\}_{k\in\nn}$ of $\{g_{i{j_k}}\}_{i\in\nn}$ verifies
\begin{equation*}
\sup\left\{\text{dist}(z,\textsf{Gr}(G_{j_k}|_K))\; : \;z\in\textsf{Gr}(g_{{i_k}{j_k}})\right\}\leq\frac{1}{2k}
\end{equation*}
for each $k\in\nn$. If we now let $\bar{g}_k\colonequals g_{{i_k}{j_k}}$ for each $k\in\nn$, then we have that
\begin{align*}
\sup\left\{\text{dist}(z,\textsf{Gr}(G|_K))\; : \;z\in\textsf{Gr}(\bar{g}_k)\right\}&\leq\sup\left\{\text{dist}(z,\textsf{Gr}(G_{j_k}|_K))\; : \;z\in\textsf{Gr}(\bar{g}_k)\right\}\\
&\quad+\sup\left\{\text{dist}(z,\textsf{Gr}(G|_K))\; : \;z\in\textsf{Gr}(G_{j_k}|_K)\right\}\leq\frac{1}{2k}+\frac{1}{2k}=\frac{1}{k},
\end{align*}
from which it follows that
\begin{equation*}
\lim_{k\to\infty}\sup\left\{\text{dist}(z,\textsf{Gr}(G|_K))\; : \;z\in\textsf{Gr}(\bar{g}_k)\right\}=0.
\end{equation*}
Hence, we have shown that $G|_K$ is a limit (in the sense of inward graph convergence) of a sequence $\{\tilde{g}_k\}_{k\in\nn}$ of single-valued continuous functions from $K$ to $\L^\d$, which concludes the proof.\qed

\section*{Acknowledgments}
We wish to express our deep gratitude to H{\'e}ctor J. Sussmann, whose unpublished work as well as several private communications with the third author have served as an important source of inspiration for this paper.

{The first author was supported by the Albanian-American Development Foundation via the Research Expertise from the Academic Diaspora (READ) Fellowship Program 2025, in the framework of the project \emph{Qualitative Analysis and Inverse Problems for a Class of Non-linear Partial Differential Equations: A Research and Training Effort}.} This work was carried out while the second author was a postdoctoral fellow at the Institute for Mathematics and its Applications (IMA) during the IMA's annual program on ``\emph{Control Theory and its Applications}''. During the preparation of this paper, the third author was a member of the Gruppo Nazionale per l'Analisi Matematica, la Probabilit{\`a} e le loro Applicazioni (GNAMPA) of the Istituto Nazionale di Alta Matematica (INdAM). He is also a member of ``INdAM - GNAMPA Project 2023'', codice CUP E53C22001930001 (``\emph{Problems with Constrained Dynamics: Non}-\emph{Smoothness and Geometric Aspects, Impulses and Delays}''), ``INdAM - GNAMPA Project 2024'', codice CUP E53C23001670001 (``\emph{Non}-\emph{Smooth Optimal Control Problems}'') and PRIN 2022, Progr-2022238YY5-PE1 (``\emph{Optimal Control Problems: Analysis, Approximations and Applications}'').

\bibliographystyle{abbrv}
\bibliography{references}

\begin{bibdiv}
\begin{biblist}

\bib{Agrachev2019}{book}{
      author={Agrachev, Andrei~A.},
      author={Barilari, Davide},
      author={Boscain, Ugo},
       title={A {C}omprehensive {I}ntroduction to {S}ub-{R}iemannian
  {G}eometry},
   publisher={Cambridge University Press},
        date={2019},
}

\bib{Angrisani2023}{article}{
      author={Angrisani, Francesca},
      author={Rampazzo, Franco},
       title={Goh conditions for minima of nonsmooth problems with unbounded
  controls},
        date={2023},
     journal={ESAIM: Control, Optimisation and Calculus of Variations},
      volume={29},
      number={17},
}

\bib{Aubin1984}{book}{
      author={Aubin, Jean-Pierre},
      author={Cellina, Arrigo},
       title={Differential {I}nclusions},
   publisher={Springer},
        date={1984},
}

\bib{Bardi1999}{article}{
      author={Bardi, Martino},
      author={Da~Lio, Francesca},
       title={On the strong maximum principle for fully nonlinear degenerate
  elliptic equations},
        date={1999},
     journal={Archiv der Mathematik},
      volume={73},
      number={4},
       pages={276\ndash 285},
}

\bib{Bardi2020}{article}{
      author={Bardi, Martino},
      author={Feleqi, Ermal},
      author={Soravia, Pierpaolo},
       title={Regularity of the minimum time and of viscosity solutions of
  degenerate eikonal equations via generalized {L}ie brackets},
        date={2020},
     journal={Set-Valued and Variational Analysis},
}

\bib{Bony1969}{article}{
      author={Bony, Jean-Michel},
       title={Principe du maximum, in{\'e}galit{\'e} de {H}arnack et
  unicit{\'e} du probl{\`e}me de {C}auchy pour les op{\'e}rateurs elliptiques
  d{\'e}g{\'e}n{\'e}r{\'e}s},
        date={1969},
     journal={Annales de L'Institut Fourier},
      volume={19},
      number={1},
       pages={277\ndash 304},
}

\bib{Botsko1985}{article}{
      author={Botsko, Michael~W.},
      author={Gosser, Richard~A.},
       title={On the differentiability of functions of several variables},
        date={1985},
     journal={The American Mathematical Monthly},
      volume={92},
      number={9},
       pages={663\ndash 665},
}

\bib{Bramanti2014}{book}{
      author={Bramanti, Marco},
       title={An {I}nvitation to {H}ypoelliptic {O}perators and
  {H}{\"o}rmander's {V}ector {F}ields},
   publisher={Springer},
        date={2014},
}

\bib{Bramanti2013}{article}{
      author={Bramanti, Marco},
      author={Brandolini, Luca},
      author={Pedroni, Marco},
       title={Basic properties of nonsmooth {H}\"ormander's vector fields and
  {P}oincar\'e's inequality},
        date={2013},
     journal={Forum Mathematicum},
      volume={25},
      number={4},
       pages={703\ndash 769},
}

\bib{Caratheodory1909}{article}{
      author={Carath{\'e}odory, Constantin},
       title={Untersuchungen {\"u}ber die grundlagen der {T}hermodynamik},
        date={1909},
     journal={Mathematische Annalen},
      volume={67},
      number={3},
       pages={355\ndash 386},
}

\bib{Cellina1970}{article}{
      author={Cellina, Arrigo},
       title={Fixed points of noncontinuous mappings},
        date={1970},
     journal={Atti della Accademia Nazionale dei Lincei. Classe di Scienze
  Fisiche, Matematiche e Naturali. Rendiconti},
      volume={49},
      number={1-2},
       pages={30\ndash 33},
}

\bib{Chow1939}{article}{
      author={Chow, Wei-Liang},
       title={{\"U}ber systeme von linearen partiellen differentialgleichungen
  erster ordnung},
        date={1939},
     journal={Mathematische Annalen},
      volume={117},
       pages={98\ndash 105},
}

\bib{Feleqi2017}{article}{
      author={Feleqi, Ermal},
      author={Rampazzo, Franco},
       title={Iterated {L}ie brackets for nonsmooth vector fields},
        date={2017},
     journal={Nonlinear Differential Equations and Applications NoDEA},
      volume={24:61},
}

\bib{Hormander1967}{article}{
      author={H{\"o}rmander, Lars},
       title={Hypoelliptic second order differential equations},
        date={1967},
     journal={Acta Mathematica},
      volume={119},
       pages={147\ndash 171},
}

\bib{Lang1993}{book}{
      author={Lang, Serge},
       title={Real and {F}unctional {A}nalysis},
   publisher={Springer},
        date={1993},
}

\bib{Montgomery2002}{book}{
      author={Montgomery, Richard},
       title={A {T}our of {S}ubriemannian {G}eometries, {T}heir {G}eodesics and
  {A}pplications},
   publisher={American Mathematical Society},
        date={2002},
}

\bib{Rampazzo2007Frobenius}{article}{
      author={Rampazzo, Franco},
       title={Frobenius-type theorems for {L}ipschitz distributions},
        date={2007},
     journal={Journal of Differential Equations},
      volume={243},
      number={2},
       pages={270\ndash 300},
}

\bib{Rampazzo2001}{inproceedings}{
      author={Rampazzo, Franco},
      author={Sussmann, H{\'e}ctor~J.},
       title={Set-valued differentials and a nonsmooth version of {C}how's
  theorem},
        date={2001},
   booktitle={Proceedings of the 40th {IEEE} {C}onference on {D}ecision and
  {C}ontrol},
       pages={2613\ndash 2618},
}

\bib{Rampazzo2007commutators}{article}{
      author={Rampazzo, Franco},
      author={Sussmann, H{\'e}ctor~J.},
       title={Commutators of flow maps of nonsmooth vector fields},
        date={2007},
     journal={Journal of Differential Equations},
      volume={232},
      number={1},
       pages={134\ndash 175},
}

\bib{Rashevskii1938}{article}{
      author={Rashevskii, Petr~K.},
       title={About connecting two points of complete non-holonomic space by
  admissible curve},
        date={1938},
     journal={Uch. Zapiski Ped. Inst. Libknexta},
      volume={2},
       pages={83\ndash 94},
}

\bib{Sussmann2000new}{incollection}{
      author={Sussmann, H{\'e}ctor~J.},
       title={New theories of set-valued differentials and new versions of the
  maximum principle of optimal control theory},
        date={2000},
   booktitle={Nonlinear control in the {Y}ear 2000 {V}olume 2 ({L}ecture
  {N}otes in {C}ontrol and {I}nformation {S}ciences)\textnormal{, A. Isidori,
  F. Lamnabhi-Lagarrigue and W. Respondek (Eds.)}},
      volume={259},
   publisher={Springer},
       pages={487\ndash 526},
}

\bib{Sussmann2000resultats}{inproceedings}{
      author={Sussmann, H{\'e}ctor~J.},
       title={R{\'e}sultats r{\'e}cents sur les courbes optimales},
        date={2000},
   booktitle={Journ{\'e}e {A}nnuelle de la {S}oci{\'e}t{\'e} {M}ath{\'e}matique
  de {F}rance},
}

\bib{Sussmann2002}{inproceedings}{
      author={Sussmann, H{\'e}ctor~J.},
       title={Warga derivative containers and other generalized differentials},
        date={2002},
   booktitle={Proceedings of the 41th {IEEE} {C}onference on {D}ecision and
  {C}ontrol},
       pages={4728\ndash 4732},
}

\bib{Sussmann2008}{incollection}{
      author={Sussmann, H{\'e}ctor~J.},
       title={Generalized differentials, variational generators, and the
  maximum principle with state constraints},
        date={2008},
   booktitle={Nonlinear and {O}ptimal {C}ontrol {T}heory ({L}ecture {N}otes in
  {M}athematics)\textnormal{, P. Nistri and G. Stefani (Eds.)}},
      volume={1932},
   publisher={Springer},
       pages={221\ndash 287},
}

\end{biblist}
\end{bibdiv}

\end{document}